\documentclass[12pt]{amsart}

\usepackage{amsthm,amssymb,amstext,amscd,amsfonts,amsbsy,amsxtra,latexsym,amsmath}
\usepackage{fullpage}
\usepackage[english]{babel}
\usepackage[latin1]{inputenc}
\usepackage{verbatim}
\usepackage{graphicx} 
\usepackage{enumitem}
\numberwithin{figure}{section}

\newcommand{\field}[1]{\mathbb{#1}}
\newcommand{\N}{\field{N}}
\newcommand{\Z}{\field{Z}}
\newcommand{\R}{\field{R}}

\newcommand{\Q}{\field{Q}}

\newcommand{\SL}{\operatorname{SL}}
\newcommand{\PSL}{\operatorname{PSL}}

\newcommand{\Res}{\operatorname{Res}}

\newcommand{\im}{\operatorname{Im}}
\newcommand{\re}{\operatorname{Re}}
\renewcommand{\H}{\mathbb{H}}

\newcommand{\slz}{\mathrm{SL}_2(\mathbb{Z})}

\newcommand{\Lpow}{r}

\newtheorem{theorem}{Theorem}[section]

\newtheorem{lemma}[theorem]{Lemma}
\newtheorem{corollary}[theorem]{Corollary}
\newtheorem{proposition}[theorem]{Proposition}

\newtheorem*{theorem*}{Theorem}

\theoremstyle{remark}
\newtheorem*{remark}{Remark}
\newtheorem*{remarks}{Remarks}
\numberwithin{equation}{section}

\renewenvironment{proof}[1][Proof]{\begin{trivlist}
\item[\hskip \labelsep {\bfseries #1:}]}{\qed\end{trivlist}}

\newcommand{\bea}{\begin{eqnarray}} 
\newcommand{\eea}{\end{eqnarray}} 
\newcommand{\be}{\begin{equation*}} 
\newcommand{\ee}{\end{equation*}} 
\newcommand{\benn}{\begin{equation}} 
\newcommand{\eenn}{\end{equation}}

\begin{document}

\title[]{Ramanujan and coefficients of meromorphic modular forms}
\author{Kathrin Bringmann} 
\address{Mathematical Institute\\University of Cologne\\ Weyertal 86-90 \\ 50931 Cologne \\Germany}
\email{kbringma@math.uni-koeln.de}
\author{Ben Kane}
\address{Department of Mathematics\\ University of Hong Kong\\ Pokfulam\\ Hong Kong}
\email{bkane@maths.hku.hk}

\date{\today}
\thanks{ The research of the first author was supported by the Alfried Krupp Prize for Young University Teachers of the Krupp foundation and the research leading to these results has received funding from the European Research Council under the European Union's Seventh Framework Programme (FP/2007-2013) / ERC Grant agreement n. 335220 - AQSER.  The research of the second author was supported by grant project numbers 27300314 and 17302515 of the Research Grants Council. 
}
\begin{abstract}
The study of Fourier coefficients of meromorphic modular forms dates back to Ramanujan, who, together with Hardy, studied the reciprocal of the weight $6$ Eisenstein series.  Ramanujan conjectured a number of further identities for other meromorphic modular forms and quasi-modular forms which were subsequently established by Berndt, Bialek, and Yee.  In this paper, we place these identities into the context of a larger family by making use of Poincar\'e series introduced by Petersson and a new family of Poincar\'e series which we construct here and which are of independent interest.  In addition we establish a number of new explicit identities.  In particular, we give the first examples of Fourier expansions for meromorphic modular form with third-order poles and quasi-meromorphic modular forms with second-order poles. 
\vspace{.05in}

\noindent
French resume:  On calcule les coefficients de Fourier de formes modulaires m\'eromorphes dans le style de Ramanujan.  
\end{abstract}
\maketitle

\section{Introduction and statement of results}
In contrast to the situation for weakly holomorphic forms,  Fourier coefficients of meromorphic modular forms have only been considered in a number of isolated special cases.  
However, Berndt, Bialek, and Yee \cite{BeBiYe} noted similarities between the coefficients of a number of examples.  Based on their observation, the goal of 
 this paper 
is to begin the search for a general structure that 
puts these special cases into a general framework, paralleling the development of a general theory for Fourier coefficients of weakly holomorphic modular forms following Ramanujan's investigation of the partition function. 

In work which gave birth to the Circle Method, Hardy and Ramanujan \cite{HR1, HR2} derived their famous asymptotic formula for the partition function $p(n)$, namely, as $n\to\infty$,
\begin{equation}\label{partitiongrowth}
p(n)\sim \frac{1}{4n\sqrt{3}}\cdot  e^{\pi \sqrt{  \frac{2n}3}  }.
\end{equation}
Rademacher \cite{Rad} then perfected  the method to derive the exact formula
\begin{equation}\label{Radformula}
p(n)= 2 \pi (24n-1)^{-\frac{3}{4}} \sum_{k =1}^{\infty}\frac{A_k(n)}{k}\cdot  I_{\frac{3}{2}}\left( \frac{\pi\sqrt{24n-1}}{6k}\right).
\end{equation}
Here $I_{\ell}(x)$ is the $I$-Bessel function of order $\ell$ and $A_k(n)$ 
denotes the Kloosterman sum
\begin{displaymath}
A_k(n):=\frac{1}{2} \sqrt{\frac{k}{12}} \sum_{\substack{h \pmod {24k}\\
h^2 \equiv -24n+1 \pmod{24k}}}  \chi_{12}(h) \cdot
e\left(\frac{h}{12k}\right),
\end{displaymath}
where $e(\alpha):=e^{2\pi i \alpha}$ and $\chi_{12}(h):=\left(\frac{12}{h}\right)$.  A key ingredient of the proof of this formula is the fact that the partition generating function is the reciprocal of a modular form with no poles in the upper half-plane.  To be more precise, the function ($q:=e^{2\pi iz}$ throughout) 
\begin{displaymath}
P(z):=\sum_{n=0}^{\infty}p(n)q^{n-\frac{1}{24}}=
q^{-\frac{1}{24}}\prod_{n=1}^{\infty}\frac{1}{1-q^{n}},
\end{displaymath}
is a weight -1/2 {\it weakly holomorphic modular form}, a meromorphic modular form whose poles (if any) are supported at cusps. Rademacher and Zuckerman \cite{RZ, Zu1, Zu2} subsequently generalized \eqref{Radformula} to obtain exact formulas for the coefficients of all weakly holomorphic modular forms of negative weight.

Much less is known about the coefficients of general meromorphic modular forms.  Hardy and Ramanujan \cite{HR3} considered the special case that the meromorphic modular form has a unique simple pole $\text{modulo} \SL_2 (\Z)$. In particular, they found a formula for the reciprocal of the weight $6$ Eisenstein series $E_6$.  Ramanujan (see pages 102--104 of \cite{RaLost}) then stated further formulas for other meromorphic functions but, as usual for his writing, did not provide a proof. His claims concerning meromorphic modular forms with simple poles were then subsequently proven by Bialek in his Ph.D. thesis written under Berndt \cite{Bi}.  Berndt, Bialek, and Yee \cite{BeBiYe} were then first to explicitly compute the Fourier coefficients of meromorphic modular forms with second-order poles, resolving the last of Ramanujan's claims about their coefficients.

  Formulas for Fourier coefficients of meromorphic modular forms look very different from \eqref{Radformula}.  To expound upon one example, we state the result for $
1/E_4$. In this case, we have
\begin{equation*}
 \frac{1}{E_4(z)} = \sum_{n=0}^{\infty} \beta_n q^n
\end{equation*}
with 
\begin{equation}\label{bn2}
 \beta_n := (-1)^n \frac{3}{E_6(\rho)}\sum_{(\lambda)}\sum_{(c,d)} \frac{h_{(c,d)}(n)}{\lambda^3} e^{\frac{\pi n \sqrt{3}}{\lambda}}.
\end{equation}
Here $\rho:=e^{\frac{\pi i}{3}}$, $\lambda$ runs over the integers of the form $\lambda = 3^a \prod_{j=1}^{r}p_j^{a_j}$
 where $a=0$ or $1$, $p_j$ is a prime of the form $6m+1$, and $a_j \in \N_0$. Moreover $(c,d)$ runs over distinct solutions to $\lambda=c^2 -cd +d^2$ and $(a,b)\in\Z^2$ is any solution to $ad-bc=1$.  We recall the definition of distinct in Section \ref{sec:firstorder}.  Finally, we let
 $h_{(1,0)}(n):=1$, $h_{(2,1)}(n):=(-1)^n$, and for $\lambda \geq 7$
\begin{equation*}
 h_{(c,d)}(n):=  2 \cos\left( (ad+bc -2ac-2bd + \lambda)\frac{\pi n}{\lambda}-6 \arctan\left(\frac{c\sqrt{3}}{2d-c}\right)\right).
\end{equation*}
\begin{remarks}
\noindent

\noindent
\begin{enumerate}[leftmargin=*]
\item
Note that $1/E_4$ has integral coefficients which can be seen by 
directly inserting the  Fourier expansion of $E_4$ and then expanding the geometric series.  However, in parallel with the right-hand side of \eqref{Radformula}, the transcendence properties of the infinite sum defining $\beta_n$ in \eqref{bn2} are in no way obvious. One concludes that for each $n\in \N_0$, the sum in the definition of $\beta_n$ is algebraic, up to division by $E_6(\rho)$.  Moreover, by the Chowla--Selberg formula \cite{ChowlaSelberg} (see the corollary to Proposition 27 of \cite{123} and the following table for the specific form used here), $E_6(\rho)=24\sqrt{3}\Omega_{\Q(\rho)}^6$, where 
$$
\Omega_{\Q(\rho)}:=\frac{1}{\sqrt{6\pi}}\left(\frac{\Gamma\left(\frac{1}{3}\right)}{\Gamma\left(\frac{2}{3}\right)}\right)^{\frac{3}{2}}.
$$
One hence obtains an explicit constant whose ratio with the sum in \eqref{bn2} is integral.
\item
Expansions like \eqref{bn2} converge extremely rapidly, giving good approximations with very few terms.  
After noting that the six terms with $\lambda=1$ are equivalent, it is easy to check that the main asymptotic growth in \eqref{bn2} comes from the $(c,d)=(1,0)$ term, yielding 
$$
\beta_n\sim (-1)^n\frac{3}{E_6(\rho)}  e^{\pi n \sqrt{3}}.
$$  
Thus the coefficients of meromorphic modular forms grow much faster than the coefficients of weakly holomorphic modular forms (compare 
with 
\eqref{partitiongrowth}).
To give an impression of the rate of convergence, $\beta_{15}$ is a 36 digit number, while the main term
 gives the first 25 digits accurately (see \cite{BeBiYe} for further examples and tables with specific data).  
\item
The coefficients $\beta_n$ may also be related to (positive weight) Poincar\'e series evaluated at $\rho$, resulting in interesting identities which were investigated by Berndt and Bialek \cite{BeBi}. 
\end{enumerate}
\end{remarks}

The identities of Bialek \cite{Bi} and Berndt, Bialek, and Yee \cite{BeBiYe} were proven using the Circle Method, while in this paper various different Poincar\'e series are used to compute the Fourier coefficients.  This yields a 
treatment for the coefficients of a wide variety of meromorphic modular forms.  Fundamentally speaking, since our method extends to meromorphic forms with poles at arbitrary points in $\H$, this indicates that all meromorphic modular forms have Fourier expansions resembling \eqref{bn2}.  We also more closely examine expansions in the shape of \eqref{bn2}, realizing \eqref{bn2} and Fourier coefficients of many other meromorphic modular forms as sums over certain ideals in the ring of integers of imaginary quadratic fields. To describe the shape of these Fourier expansions, for $K=\Q(\mathfrak{z})$ with $\mathfrak{z}=\rho$ or $\mathfrak{z}=i$, we denote the norm in $\mathcal{O}_K$ by $N$ and the sum over all primitive ideals of $\mathcal{O}_K$ by $\sum_{\mathfrak{b}\subseteq\mathcal{O}_K}^*$. We further denote by $\omega_{\mathfrak{z}}:=\#\Gamma_{\mathfrak{z}}$ the size of the stabilizer $\Gamma_{\mathfrak{z}}\subset\PSL_2(\Z)$.  
We also require a function $C_{m}\left(\mathfrak{b},n\right)$ which closely resembles $h_{c,d}(n)$ for $\mathfrak{b}:=(c\mathfrak{z}+d)$.  Explicit definitions are given in Section \ref{sec:firstorder}.  Using this notation, we may rewrite Hardy and Ramanujan's formula for $1/E_6$ as 
$$
\frac{1}{E_6(z)} = \frac{\omega_{i}}{E_4^2(i)} \sum_{n=0}^{\infty} \;\sideset{}{^*}\sum_{\mathfrak{b}\subseteq\mathcal{O}_{\Q(i)}}\frac{C_{8}\left(\mathfrak{b},n\right)}{N(\mathfrak{b})^{4}} e^{\frac{2\pi n}{N(\mathfrak{b})}} q^n.
$$

 Berndt, Bialek, and Yee \cite{BeBiYe} investigated a beautiful relation between the coefficients of $1/E_6$ and the coefficients of weight $-4$ meromorphic modular forms with second-order poles at $z=i$. In particular, they obtained 
\begin{equation}\label{eqn:E4^2/E6^2}
\frac{E_4^2(z)}{E_6^2(z)}= \frac{{2\omega_i}}{E_4^2(i)}\sum_{n=0}^{\infty}\; \sideset{}{^*}\sum_{\mathfrak{b}\subseteq\mathcal{O}_{\Q(i)}}\left(n + \frac{3N(\mathfrak{b})}{2\pi}\right)\frac{C_{8}\left(\mathfrak{b},n\right)}{N(\mathfrak{b})^{4}} e^{\frac{2\pi n}{N(\mathfrak{b})}} q^n.
\end{equation}

It is striking that all of the meromorphic modular forms encountered above 
have Fourier expansions which may be written as linear combinations of the series 
(with $\mathfrak{z}_2:=\im(\mathfrak{z})$) 
\begin{equation}\label{Fklr}
F_{k,\ell,r}(q)=F_{k,\ell,r}(\mathfrak{z};q):=\sum_{n=0}^{\infty} \;\sideset{}{^*}\sum_{\mathfrak{b}\subseteq\mathcal{O}_{\Q(\mathfrak{z})}} \frac{C_{k}\left(\mathfrak{b},n\right)}{N(\mathfrak{b})^{\frac{k}{2}-\ell}} n^r e^{\frac{2\pi n}{N(\mathfrak{b})}\mathfrak{z}_2} q^n.
\end{equation}
The appearance of these functions in the above formulas for meromorphic modular forms is not isolated.  Ramanujan found similar formulas for powers of the weight $2$ Eisenstein series multiplied by meromorphic modular forms, which were later proven by Bialek \cite{Bi}.
To state two examples (again rewritten in terms of ideals), Bialek showed (see Theorem 2.12 of \cite{Bi})
\begin{equation}\label{eqn:E2/E6coeff}
\frac{E_2(z)}{E_6(z)} = \frac{\omega_i}{E_4^2(i)} \sum_{n=0}^{\infty} 
\;\sideset{}{^*}\sum_{\mathfrak{b}\subseteq\mathcal{O}_{\Q(i)}} \frac{C_{8}\left(\mathfrak{b},n\right)}{N(\mathfrak{b})^{3}} e^{\frac{2\pi n}{N(\mathfrak{b})}} q^n,
\end{equation}
\begin{equation}\label{eqn:E2^2/E6coeff}
\frac{E_2^2(z)}{E_6(z)} = \frac{\omega_i}{E_4^2(i)} \sum_{n=0}^{\infty} 
\;\sideset{}{^*}\sum_{\mathfrak{b}\subseteq\mathcal{O}_{\Q(i)}} \frac{C_{8}\left(\mathfrak{b},n\right)}{N(\mathfrak{b})^{2}} e^{\frac{2\pi n}{N(\mathfrak{b})}} q^n.
\end{equation}
The similarity between the Fourier coefficients in \eqref{eqn:E4^2/E6^2}, \eqref{eqn:E2/E6coeff}, and \eqref{eqn:E2^2/E6coeff} leads one to question whether the coefficients of all meromorphic (quasi-)modular forms have this shape.  In this paper, certain Poincar\'e series (see Section \ref{sec:firstorder}) considered by Petersson \cite{Pe1} and a new family of Poincar\'e series, introduced in Section \ref{sec:HigherOrder}, are employed to show that the formulas \eqref{eqn:E4^2/E6^2}, \eqref{eqn:E2/E6coeff}, and \eqref{eqn:E2^2/E6coeff} all fit into a general family of identities.  To describe the result, write $z=x+iy\in\H$ and denote by $\mathcal{M}_{\mathfrak{z},\nu}$ the space of negative weight meromorphic modular forms whose only poles modulo $\SL_2(\Z)$ are at the point $z=\mathfrak{z}$ and have order at most $\nu$. 
\begin{theorem}\label{thm:mainmero}
If $\mathfrak{z}\in \{i,\rho\}$ and $y>\mathfrak{z}_2$, then every $F\in \mathcal{M}_{\mathfrak{z},3}$ has a Fourier expansion which is  a linear combination of the expansions $F_{2k\omega_{\mathfrak{z}},\ell,r}
$ with $k\in \N$, $\ell\in \N_0$, and $r\in \N_0$.
\end{theorem}
\begin{remarks}
\noindent

\noindent
\begin{enumerate}[leftmargin=*]
\item
By writing meromorphic forms as linear combinations of the images of 
certain differential operators on Poincar\'e series, Petersson obtained a formula for the coefficients of meromorphic modular forms (see (4b.9) of \cite{Pe1}) as sums of derivatives of Poincar\'e series.  However, to obtain an explicit formula for the coefficients, one  has   to take derivates of the Fourier expansions of these Poincar\'e series, yielding a formula which has infinite sums of infinite sums of Kloosterman sums.  To obtain formulas resembling those of Hardy and Ramanujan \cite{HR3}, one has to explicitly compute the action of the differential operators on 
these Poincar\'e series.  This is essentially the method undertaken in this paper.
\item
The method used here seems to generalize to meromorphic modular forms with poles of arbitrary order.  
\end{enumerate}
\end{remarks}

In individual cases, 
the methods in this paper also give the implied constants in Theorem \ref{thm:mainmero}  (see \eqref{eqn:explicitconsts1} and \eqref{eqn:explicitconsts2} for the exact constants), yielding a number of identities closely resembling \eqref{eqn:E4^2/E6^2}.
\begin{corollary}\label{cor:explicit}
\noindent

\noindent
\begin{enumerate}[leftmargin=*]
\item[\rm(1)]
If $f\in \{ 1/E_6^2, E_4/E_6^2, 1/E_4^2, E_6/E_4^2\}$ and  $\mathfrak{z}\in\left\{i,\rho\right\}$ is chosen so that $f\in \mathcal{M}_{\mathfrak{z},2}$, then
there exist constants $k=k_f$, $a=a_f$, and $c=c_f$ such that we have, for $y>\mathfrak{z}_2$, 
$$
f(z) = \sum_{n=0}^{\infty} \;\sideset{}{^*}\sum_{\mathfrak{b}\subseteq\mathcal{O}_{\Q(\mathfrak{z})}} \frac{C_{k}\!\left(\mathfrak{b},n\right)}{N(\mathfrak{b})^{\frac{k}{2}}}\left(a N(\mathfrak{b})+ c n\right)e^{\frac{2 \pi n}{N(\mathfrak{b})}\mathfrak{z}_2} q^n.
$$
\item[\rm(2)]
If $f\in \{1/E_4^3, E_6/E_4^3\}$, then there exist explicit constants $k=k_f$, $a=a_f$, $c=c_f$, and $d=d_f$ such that we have for $y>
\sqrt{3}/2
$
$$
f(z) =  \sum_{n=0}^{\infty} \;\sideset{}{^*}\sum_{\mathfrak{b}\subseteq\mathcal{O}_{\Q(\rho)}} \frac{C_{k}\!\left(\mathfrak{b},n\right)}{N(\mathfrak{b})^{\frac{k}{2}}}\left(a N(\mathfrak{b})^2+ c N(\mathfrak{b}) n+ dn^2\right)e^{\frac{\sqrt{3} \pi n}{N(\mathfrak{b})}} q^n.
$$
\end{enumerate}
\end{corollary}
\begin{remark}
Corollary \ref{cor:explicit} (2) gives the first examples for Fourier coefficients 
of meromorphic modular forms with 
third-order poles.
\end{remark}
Denote by $\mathcal{M}_{\kappa, \mathfrak{z}, \nu}$ the subspace of $\mathcal{M}_{\mathfrak{z},\nu}$ consisting of forms satisfying weight $\kappa\in\Z$ modularity.  The following theorem generalizes formulas like \eqref{eqn:E2/E6coeff} and \eqref{eqn:E2^2/E6coeff} as well as a formula for $\frac{E_2}{E_4}$ obtained by Bialek \cite{Bi}.

\begin{theorem}\label{thm:mainE2}
Suppose that $\mathfrak{z}\in \{i,\rho\}$, $m\geq 4+2\ell$, $\ell\in\mathbb{N}$, and $1\leq \nu\leq 2$.  If $F\in \mathcal{M}_{2-m,\mathfrak{z},\nu}$, then, for $y>\mathfrak{z}_2$, the function $E_2^{\ell}F$ has a Fourier expansion which is a linear combination of $F_{2k\omega_{\mathfrak{z}},\ell,r}$ with $k\in \N$, $\ell\in \N_0$, and $r\in \N_0$.
\end{theorem}
\begin{remark}
Note that by the valence formula, the restriction $1\leq \nu\leq 2$ also implies a restriction on $m$.  One may then check that, under the restrictions on $m$ and $\nu$, the space of weight $m+2(\nu-1)$ cusp forms is trivial.  
\end{remark}
Comparing the cases covered by Theorem \ref{thm:mainE2} with the known special cases leads to the following new explicit example (see \eqref{eqn:explicitconsts3} for the explicit constants).

\begin{corollary}\label{cor:E2explicit}
There exist explicit constants $k$, $a$, $c$, and $d$ such that for $y>1$, we have
$$
\frac{E_2(z) E_4^2(z)}{E_6^2(z)}=\sum_{n=0}^{\infty} \;\sideset{}{^*}\sum_{\mathfrak{b}\subseteq\mathcal{O}_{\Q(i)}} \frac{C_{k}\!\left(\mathfrak{b},n\right)}{N(\mathfrak{b})^{\frac{k}{2}-1}}\left(a N(\mathfrak{b})+ c n\right)e^{\frac{2 \pi n}{N(\mathfrak{b})}} q^n  + d\sum_{n=0}^{\infty} \;\sideset{}{^*}\sum_{\mathfrak{b}\subseteq\mathcal{O}_{\Q(i)}} \frac{C_{k-4}\!\left(\mathfrak{b},n\right)}{N(\mathfrak{b})^{\frac{k}{2}-2}}e^{\frac{2 \pi n}{N(\mathfrak{b})}} q^n.
$$

\end{corollary}
\begin{remark}
Corollary \ref{cor:E2explicit} contains the first case where $E_2$ is multiplied by a meromorphic form which has a pole that is not simple.
\end{remark}
As alluded to earlier, the method used here may also be applied to meromorphic modular forms with poles at arbitrary points in $\H$.  To give the flavor of the resulting formulas, 
we compute the Fourier coefficients of one infinite family of weight $-8$ meromorphic modular forms with poles at arbitrary points in $\H$.  For $\tau_0\in \H$ not equivalent to $\rho$ or $i$ modulo $\SL_2(\Z)$, we hence define
\begin{equation}\label{eqn:Ftaudef}
F_{\tau_0}(z):=\frac{E_4(z)}{\Delta(z)\left(j(z)-j\left(\tau_0\right)\right)^2},
\end{equation}
where $\Delta$ is the discriminant function and $j$ is the $j$-invariant.  We see later that each $F_{\tau_0}$ is a weight $-8$ meromorphic modular form whose only pole modulo $\SL_2 (\Z)$ is at $\tau_0$ and has order exactly $2$.

The coefficients of $F_{\tau_0}$ closely resemble \eqref{bn2}, 
where the sum runs over $(c,d)=1$ (with $\lambda$ satisfying $\lambda = \left|c\tau_0+d\right|^2$).  The $(c,d)$th contribution to the $n$th coefficient is written 
as a linear combination 
 of $B_{m,c,d}\left(\tau_0,n\right)$ (defined in \eqref{eqn:Bdef}), where $m\in 2\N$.  
The coefficients in the linear combination depend on $n$, $\lambda$, and the 
Laurent 
coefficients
 of the principal part of $F_{\tau_0}$ around $z=\tau_0$,
 which we denote by $\lambda_{-2}$ and $\lambda_{-1}$ (these are explicitly computed in \eqref{eqn:deflambda1} and \eqref{eqn:deflambda2}).  
Denoting $v_0:=\im(\tau_0)$, we
 are now ready to state the result.
\begin{theorem}\label{thm:PoleArbitrary}
If $y>\im\left(M\tau_0\right)$ for all $M\in \SL_2(\Z)$, then
\begin{multline*}
F_{\tau_0}(z) = 2\pi i \sum_{n=0}^{\infty}\sum_{\substack{c,d\in \Z\\ (c,d)=1}}\Bigg(\frac{i\lambda_{-2}}{4}\left( \frac{5\left|c\tau_0+d\right|^2}{v_0} + 2\pi n \right)B_{12,c,d}\left(\tau_0,n\right)\\
 + \left( \frac{\lambda_{-2}}{40\left(2iv_0\right)^{10}}-\frac{\lambda_{-1}}{2}\right)B_{10,c,d}\left(\tau_0,n\right)\Bigg)q^n.
\end{multline*}
\end{theorem}
\begin{remark}
If $\tau_0$ generates an imaginary quadratic field of class number one, then the $n$th coefficient of $F_{\tau_0}$ can be explicitly written as a sum over ideals in a shape similar to \eqref{eqn:E4^2/E6^2}.  Since the proof closely follows the proof of Theorem \ref{thm:mainmero}, we do not work out the details here.
\end{remark}

The paper is organized as follows. In Section \ref{sec:firstorder}, we recall Poincar\'e series introduced by Petersson with simple poles and rewrite their Fourier expansions in terms of sums over primitive ideals if the poles are at the special points $
\mathfrak{z}
\in\{i,\rho\}$.  In Section \ref{sec:HigherOrder}, we 
recall the construction of functions with higher-order poles and prove Theorem \ref{thm:mainmero} by relating 
these functions to new Poincar\'e series which we explicitly construct.  In Section \ref{sec:quasi}, we prove Theorem \ref{thm:mainE2} by rewriting powers of $E_2$ times the new Poincar\'e series from the previous section.  In Section \ref{sec:examples}, we consider explicit examples and compute the relevant implied constants to prove Corollaries \ref{cor:explicit} and \ref{cor:E2explicit}.  Finally, in Section \ref{sec:InfFamily}, we determine the Fourier expansion of the infinite family of meromorphic modular forms $F_{\tau_0}$, proving Theorem \ref{thm:PoleArbitrary}.

\section*{Acknowledgements}  
The authors would like to thank Pavel Guerzhoy and Steffen L\"obrich for helpful comments on earlier versions of the paper.

\section{Meromorphic Poincar\'e series with simple poles}\label{sec:firstorder}
 \subsection{Poincar\'e series of Petersson}
We start by recalling certain meromorphic weight $m\in 2\N$ Poincar\'e series, introduced in (2b.9) of \cite{Pe1},
\begin{equation*}
H_{m}(\mathfrak{z},z) := 2 \pi i \sum_{M \in \Gamma_{\infty}\backslash \textrm{SL}_2(\Z)} \frac{1}{1-e^{2 \pi i (z-\mathfrak{z})}}\Bigg|_{m,\mathfrak{z}} M,
\end{equation*}
 where $\Gamma_{\infty}:= \left\lbrace 
\left( \begin{smallmatrix}1 & \ell \\ 0 & 1 \end{smallmatrix}\right),  \ell \in  \Z \right\rbrace$ and $|_{m,\mathfrak{z}} $ denotes the weight $m$ slash operator with respect to $\mathfrak{z}$.  
\begin{remark}
Petersson used the notation $H_{-m}$ because in his time weight $m$ modular forms were referred to as having dimension $-m$.  
For certain related functions that Petersson called $Y_{m-2}$, we similarly write $Y_{2-m,\nu}$ in \eqref{eqn:Ydef} below. 
\end{remark}
The sum defining $H_m(\mathfrak{z},z)$ converges absolutely for $m\geq 4$ and $\mathfrak{z}\mapsto H_m(\mathfrak{z},z)$ 
is then a meromorphic modular form of weight $m$ by construction.  It has at most a simple pole at $\mathfrak{z}=Mz$ ($M=\left( \begin{smallmatrix}a & b \\c & d \end{smallmatrix}\right)\in \SL_2(\Z)$) with (see (2b.11) of \cite{Pe1}) 
\begin{equation}\label{res}
\Res_{\mathfrak{z} = Mz} \left(  H_{m} \left( \mathfrak{z}, z\right) \right) = \varepsilon (z) (cz+d)^{m-2},\ \text{where}
\end{equation}
$$
\varepsilon (z) : = 
\begin{cases}
2\omega_{z} & \text{if }\frac{m}{2}\equiv 1\pmod{\omega_{z}},\\
0&\text{otherwise.}
\end{cases}
$$

\noindent
As further noted by Petersson \cite{Pe1}, the functions $H_m(\mathfrak{z},z)$ vanish as $\mathfrak{z}\to i\infty$ (see also Lemma \ref{lem:Hgengrowth} (1) below) and we refer to such meromorphic modular forms as \begin{it}meromorphic cusp forms.\end{it}

The cases when linear combinations of $z\mapsto H_m(\mathfrak{z},z)$ are modular of weight $2-m$ were classified in Satz 1 of \cite{Pe1}. 
This yields an explicit constructuction of all 
\rm
meromorphic modular forms of negative weight with at most simple poles.  In particular, for fixed $\tau_0\in \H$, $H_m\left(\tau_0,z\right)$ satisfies weight $2-m$ modularity precisely when all cusp forms of weight $m$ vanish at $\tau_0$.  The residue at $z=M\mathfrak{z}$ is given by (see the discussion following (3a.10) of \cite{Pe1})
\begin{equation}\label{eqn:zRes}
\Res_{z=M\mathfrak{z}}\left( H_{m}\left(\mathfrak{z},z\right) \right)=\widetilde{\varepsilon}(\mathfrak{z})\left(c\mathfrak{z}+d\right)^{-m},
\end{equation}
where
$$
\widetilde{\varepsilon}(\mathfrak{z}):=
\begin{cases}
-2\omega_{\mathfrak{z}}&\text{if }\frac{m}{2}\equiv 0\pmod{\omega_{\mathfrak{z}}},\\
0&\text{otherwise}.
\end{cases}
$$
\subsection{Relation to imaginary quadratic fields}\label{sec:imagquad}
Our first application of the functions $H_m(\mathfrak{z},z)$ is to give an alternative interpretation of the meaning of distinct solutions in the Fourier expansions 
of  Hardy and Ramanujan.  In the sum \eqref{bn2}, we call a solution  $\left(
c_2, d_2\right)$ 
of $\lambda=c_2^2-c_2d_2+d_2^2$ 
 \begin{it}equivalent\end{it} to $(c,d)$ if it is one of the following:
$$
\pm (c,d),\  \pm (d,c),\ \pm (c-d,c),\ 
\pm (c,c-d),\ \pm (d,d-c),\ \pm (c-d,-d);
$$
we say that the solutions are \begin{it}distinct\end{it} otherwise.  The key step is to rewrite the expansion of $H_{m}(\mathfrak{z},z)$ at the special points $i$ and $\rho$. To state the result, we require some further notation.  We write throughout $\mathfrak{z}=\mathfrak{z}_1+i\mathfrak{z}_2$ and $z=x+iy$ with $\mathfrak{z}_1,\mathfrak{z}_2,x,y\in\R$.
If $\mathfrak{z}$ lies in an imaginary quadratic field $K$, we use the notation $\mathcal{O}_K$ for the ring of integers of $K$ and write ideals of $\mathcal{O}_K$ as $\mathfrak{b}\subseteq \mathcal{O}_K$.  We call the ideals which are not divisible by any principal ideal $(g)$ with $g\in \Z$ \begin{it}primitive\end{it} and denote the sum over all primitive ideals of $\mathcal{O}_K$ by $\sum_{\mathfrak{b}\subseteq\mathcal{O}_K}^*$.  For $\gamma=c\mathfrak{z}+d\in \mathcal{O}_{\Q(\mathfrak{z})}$ 
and $n\in \N_0$, we furthermore define the root of unity
$$
A_m\left(\gamma,n\right):=e\left(-\frac{n}{N(\gamma)}\left(ac|\mathfrak{z}|^2+bd+\mathfrak{z}_1(ad+bc)\right)-\frac{m}{2\pi}
\arg(\gamma)\right),
$$
where $N(\gamma)$ is the norm in $\mathcal{O}_K$ (we also use this notation for norms of ideals).  For the principal ideal $\mathfrak{b}=(\gamma)$, we let $A_m(\mathfrak{b},n):= A_m(\gamma,n)$.  The proof of Proposition \ref{realFourier} below shows that this definition is independent of the choice of a generator for $\mathfrak{b}$.  
Following (2.2.44) of \cite{Bi} (for the case $m=1$), one sees that
$$
\re\left(A_{6m}\left(c\rho+d,n\right)\right) = C_{6m}\left((c\rho+d),n\right),
$$
where for $\mathfrak{b}=(c\rho+d)\subset\mathcal{O}_{\Q(\rho)}$ we define
$$
C_{6m}\left(\mathfrak{b},n\right) := 
(-1)^n
\cos\left(\frac{\pi n}{N(\mathfrak{b})}\left(ad+bc-2ac-2bd\right) + \pi n -6m\arctan\left(\frac{c\sqrt{3}}{2d-c}\right)\right).
$$
Here $a,b\in \Z$ are any choice for which $ad-bc=1$. 
Similarly,  
for $\mathfrak{b}=(ci+d)\subseteq\mathcal{O}_{\Q(i)}$, we let
$$
C_{4m}\left(\mathfrak{b},n\right):=\cos\left(\frac{2\pi n}{N(\mathfrak{b})}\left(ac+bd\right)+4m\arctan\left(\frac{c}{d}\right)\right),
$$
so that (2.3.35) of \cite{Bi} yields
$$
\re\left(A_{4m}\left(ci+d,n\right)\right) = C_{4m}\left((ci+d),n\right).
$$

For $\mathfrak{z}=i$ or $\mathfrak{z}=\rho$, the above formulas lead to an evaluation of the Fourier expansion of $z\mapsto H_m(\mathfrak{z},z)$ in terms of ideals.
\begin{proposition} \label{realFourier}
For every $m\in \N$,  $\mathfrak{z}\in\{i,\rho\}$, and $y>\mathfrak{z}_2$, one has 
$$
H_{2m\omega_{\mathfrak{z}}
}(\mathfrak{z},z) = 4\pi i \omega_{\mathfrak{z}} \sum_{n=0}^{\infty}\; \sideset{}{^*}\sum_{\mathfrak{b}\subseteq\mathcal{O}_{\Q(\mathfrak{z})}} \frac{C_{2m\omega_\mathfrak{z}}}{\left(\mathfrak{b},n\right)}{N(\mathfrak{b})^{m\omega_{\mathfrak{z}}}}e^{\frac{2\pi n\mathfrak{z}_2}{N(\mathfrak{b})}} q^n.
$$
\end{proposition}
Before proving Proposition \ref{realFourier}, we first write the Fourier coefficients of $H_m$ in a preliminary shape similar to \eqref{bn2}.
\begin{lemma}\label{firstFourier}
If $y>\im\left(M\mathfrak{z}\right)$ for all $M\in \SL_2(\Z)$, then, for $m\in\N$  with $m>2$,
$$
H_{m}\left(\mathfrak{z},z\right)=2\pi i \sum_{n=0}^{\infty} \sum_{\substack{c,d\in \Z\\ (c,d)=1}} \frac{e^{\frac{2\pi n\mathfrak{z}_2}{\left|c\mathfrak{z}+d\right|^2}} }{\left(c\mathfrak{z}+d\right)^m}e\left(-\frac{n}{\left|c\mathfrak{z}+d\right|^2}\left(ac|\mathfrak{z}|^2 + bd +\mathfrak{z}_1\left(ad+bc\right)\right)\right)q^n.
$$
\end{lemma}

\begin{proof}
Since $y>\im\left(M\mathfrak{z}\right)$, (3a.4) and (3a.7) of \cite{Pe1} imply that 
\begin{equation}\label{eqn:Hexp}
H_{m}\left(\mathfrak{z}, z\right) =2\pi i\sum_{n=0}^{\infty}\sum_{M\in \Gamma_{\infty}\backslash \SL_2(\Z)} \frac{e^{-2\pi i n M\mathfrak{z}}}{\left(c\mathfrak{z}+d\right)^{m}} q^n.
\end{equation}
The claim of the lemma follows by a direct calculation showing that
$$
e^{2\pi i M\mathfrak{z}}  = e^{-\frac{2\pi \mathfrak{z}_2}{\left|c\mathfrak{z}+d\right|^2}}e\left(\frac{1}{\left|c\mathfrak{z}+d\right|^2}\left(ac|\mathfrak{z}|^2 + bd +\mathfrak{z}_1\left(ad+bc\right)\right)\right).
$$
\end{proof}

\begin{proof}[Proof of Proposition \ref{realFourier}]
Since the arguments for $\mathfrak{z}=i$ and $\mathfrak{z}=\rho$ are analogous, we only consider the case $\mathfrak{z}=i$.  By a direct calculation,
we see that, for $M\in\SL_2 (\Z)$,
$$
\left[\frac{1}{1-e^{2\pi i (z-\mathfrak{z})}}\Big|_{m,\mathfrak{z}} MS\right]_{\mathfrak{z}=i} 
= i^{-m} \left[\frac{1}{1-e^{2\pi i (z-\mathfrak{z})}}\Big|_{m,\mathfrak{z}} M \right]_{\mathfrak{z}=i},
$$
with $S:= \left(\begin{smallmatrix}0&-1\\1&0\end{smallmatrix}\right)$.
Hence, if $4\mid m$, one may write 
\begin{equation}\label{eqn:HmodS}
H_{m}(i,z) = 4\pi i\omega_i \left[\sum_{M\in\Gamma_\infty\setminus \slz/<S>}\frac{1}{1-e^{2\pi i(z-\mathfrak{z})}}\Bigg|_{m,\mathfrak{z}} M\right]_{\mathfrak{z}=i}.
\end{equation}

Since right multiplication by $S$ sends $(c,d)\to (d,-c)$, it is natural to consider the corresponding equivalence relation on $\Z^2$ 
$$
(c,d)\sim (d,-c)\sim (-c,-d)\sim (-d,c).
$$
We then expand $H_m(i, z)$ as in \eqref{eqn:Hexp} and follow the proof of Lemma \ref{firstFourier} to obtain that 
$$
H_{m}(i,z) = 4\pi i\omega_i \sum_{n=0}^{\infty} \sum_{\substack{ (c,d)=1\\ (c,d)\in \Z^2/\sim}} B_{m,c,d}(i,n) q^n,
$$
where the $(c,d)$ contribution to the $n$th coefficient is given by
\begin{equation}\label{eqn:Bdef}
B_{m,c,d}(\mathfrak{z},n):=\frac{1}{\left(c\mathfrak{z}+d\right)^m}e^{\frac{2\pi n \mathfrak{z}_2}{\left|c\mathfrak{z}+d\right|^2}} e\left(-\frac{n}{\left|c\mathfrak{z}+d\right|^2}\left(ac|\mathfrak{z}|^2 + bd +\mathfrak{z}_1\left(ad+bc\right)\right)\right).
\end{equation}

We now rewrite the sum over $(c,d)\in \Z^2/\sim$.  We note that there is a natural correspondence between $(c,d)\in \Z^2$ and $ci+d\in \mathcal{O}_K$ with $K:=\Q(i)$.  Under this correspondence, the relation $\sim$ on $\Z^2$ corresponds to multiplication by units in $\mathcal{O}_K$.  It follows that the inner sum may be written as the sum over all (principal) ideals $\mathfrak{b}=\mathfrak{b}_{c,d}:=(ci+d)\subseteq\mathcal{O}_K$ with $(c,d)=1$.  Furthermore, 
since $\mathcal{O}_K$ is a PID, every ideal is of the form $\mathfrak{b}_{c,d}$ for some $c,d\in \Z$ and the restriction $(c,d)=1$ is equivalent to 
$\mathfrak{b}$ being a primitive ideal.  
The inner sum may therefore be written as $
\sum_{\mathfrak{b}\subseteq\mathcal{O}_K}^*
$ using the notation above.  In particular it follows that $B_{m,c,d}(i,n)$ is independent of the choice of generator for the ideal $\mathfrak{b}_{c,d}$.  From this, one may also conclude that $A_{m}(\mathfrak{b},n)$ is well-defined.  Indeed, since
$$
\lambda(c,d)=\lambda_{\mathfrak{z}}(c,d):=\left|c \mathfrak{z} +d\right|^2 = N\left(\mathfrak{b}\right),
$$
we see by rewriting 
\begin{equation}\label{eqn:BArewrite}
B_{m,c,d}(i,n)= \frac{e^{\frac{2\pi n}{N\left(\mathfrak{b}\right)}}}{N\left(\mathfrak{b}\right)^{\frac{m}{2}}}A_{m}\left(ci+d,n\right),
\end{equation}
that $A_{m}\left(ci+d,n\right)$ also only depends on the ideal $\mathfrak{b}$, verifying that $A_m\left(\mathfrak{b},n\right)$ is well-defined.  

It remains to rewrite $A_{m}\left(\mathfrak{b},n\right)$.  To do so, we combine the terms $(c,d)$ and $(d,c)$.  For this, note that 
$$
\left(\overline{ci+d}\right)^{4m}=\left(di+c\right)^{4m},\ \text{while}
$$ 
$$
\overline{e\left(-\frac{n}{\lambda(c,d)}\left(ac+ bd +\mathfrak{z}_1\left(ad+bc\right)\right)\right)}=e\left(\frac{n}{\lambda(d,c)}\left(ac +bd + \mathfrak{z}_1\left(ad+bc\right)\right)\right).
$$
The terms $(c,d)$ and $(d,c)$ hence contribute together $2\re\left(B_{m,c,d}(i,n)\right)$. 
Plugging this into \eqref{eqn:BArewrite} yields the statement of the proposition.
\end{proof}

\section{Higher-order poles and the proof of Theorem \ref{thm:mainmero}}\label{sec:HigherOrder}
In this section, we determine the Fourier expansions of meromorphic modular forms with second-order and third-order poles.  Combining this with the results from Secton \ref{sec:firstorder}, we prove Theorem \ref{thm:mainmero}.  Here the situation is more complicated because 
the  Poincar\'e series $
H_m$ do not suffice anymore.

\subsection{Poincar\'e series with higher-order poles}\label{HigherOrder}
Petersson investigated (see Section 4 of \cite{Pe1}) the modularity of derivatives of $H_{m}$ with respect to $\mathfrak{z}$ in order to construct meromorphic modular forms with higher-order poles in $z$.  
For $\mathfrak{z},z\in \H$, $m\in 2\N$, and $\nu\in -\N$, the relevant functions are defined by \begin{equation}\label{eqn:Ydef}
Y_{2-m,\nu}\left(\mathfrak{z},z\right):=\frac{1}{(-\nu-1)!}\frac{\partial^{-\nu-1}}{\partial X_{\mathfrak{z}}(\alpha)^{-\nu-1}}\left[\left(\alpha-\overline{\mathfrak{z}}\right)^m H_{m}\left(\alpha,z\right)\right]_{\alpha=\mathfrak{z}},
\end{equation}
where $X_\mathfrak{z}(\alpha):=\frac{\alpha-\mathfrak{z}}{\alpha-\overline{\mathfrak{z}}}$.  
These functions 
were used to classify weight $2-m$ meromorphic modular forms by their principal parts.  However, they are insufficient for our purposes since one cannot easily determine their Fourier coefficients.  For this reason, we construct another family of Poincar\'e series whose Fourier coefficients may be more easily computed.  For $\ell\in \N_0$, we hence formally define (recall that $\mathfrak{z}=\mathfrak{z}_1+i\mathfrak{z}_2$)

\begin{equation}\label{eqn:Htildegendef}
H_{m,\ell}(\mathfrak{z},z):=2\pi i \sum_{M\in\Gamma_\infty \backslash \SL_2(\Z)}\frac{\mathfrak{z}_2^{-\ell} }{1-e^{2\pi i(z-\mathfrak{z})}}\Bigg|_{m, \mathfrak{z}} M.
\end{equation}
Following the  proof of the convergence of $H_{m}(\mathfrak{z},z)$ 
in \cite{Pe1}, the functions $H_{m,\ell}(\mathfrak{z},z)$ converge absolutely if $m\geq 4+2\ell$ and are modular of weight $m$ in the $\mathfrak{z}$ variable by construction.  They also form a family of functions that (essentially) map to each other under the action of the \begin{it}Maass lowering operator\end{it} $L=L_{\mathfrak{z}}:=-2i\mathfrak{z}_2^2\frac{\partial}{\partial \overline{\mathfrak{z}}}. $ In particular, we have
\begin{equation}\label{lowerH}
L_\mathfrak{z}\left(H_{m, \ell}(\mathfrak{z}, z)\right)=-\ell H_{m-2, \ell-1}(\mathfrak{z}, z).
\end{equation}

The Fourier coefficients of the functions $H_{m,\ell}(\mathfrak{z},z)$ and their derivatives with respect to $z$ (slightly abusing notation for the partial derivative and omitting $\ell$ when it is zero)
$$
H_{m,\ell}^{(\Lpow)}\left(\mathfrak{z},z\right):=\frac{\partial^\Lpow}{\partial z^\Lpow}H_{m,\ell}\left(\mathfrak{z},z\right)
$$
are closely related to the coefficients of $H_{m}$ given in 
Proposition \ref{realFourier}. 
\begin{theorem}\label{genFC}
\noindent

\noindent
\begin{enumerate}[leftmargin=*]
\item[\rm(1)]
If $\ell,\Lpow\in \N_0$, $m\geq 4+2\ell$ and $y>\im(M\mathfrak{z})$ holds for every $M\in \SL_2(\Z)$, then 
$$
H_{m,\ell}^{(\Lpow)}\left(\mathfrak{z},z\right)=2\pi i \sum_{n=0}^{\infty} \sum_{\substack{c,d\in \Z\\ (c,d)=1}} \left(\frac{\lambda(c,d)}{\mathfrak{z}_2}\right)^{\ell}\left(2\pi i n\right)^\Lpow B_{m,c,d}(\mathfrak{z},n) q^n.
$$
\item[\rm(2)]
For every $\mathfrak{z}\in \left\{ i,\rho\right\}$, $\ell,\Lpow\in\N_0$, $m\geq \frac{4+2\ell}{k}$ and $y>\mathfrak{z}_2$, one has 
$$
H_{2m\omega_{\mathfrak{z}},\ell}^{(\Lpow)}\left(\mathfrak{z},z\right) = 4\pi i \omega_{\mathfrak{z}} \sum_{n=0}^{\infty} \;\sideset{}{^*}\sum_{\mathfrak{b}\subseteq\mathcal{O}_{\Q(\mathfrak{z})}}\frac{C_{{2m\omega_{\mathfrak{z}}}}\left(\mathfrak{b},n\right)}{N(\mathfrak{b})^{{m\omega_{\mathfrak{z}}}}}
\left(\frac{N(\mathfrak{b})}{\mathfrak{z}_2}\right)^\ell \left(2\pi i n\right)^{\Lpow}e^{\frac{2\pi n \mathfrak{z}_2}{N(\mathfrak{b})}} q^n.
$$

\end{enumerate}
\end{theorem}
\begin{remark}
For $\mathfrak{z}\in \{i,\rho\}$, Theorem \ref{genFC} (2) shows that 
the Fourier expansion of 
$H_{mk,\ell}^{(r)}(\mathfrak{z},z)$ is a constant multiple of $F_{mk,\ell,r}$, defined in (\ref{Fklr}). 
\end{remark}
\begin{proof}
\noindent

\noindent
(1) We expand $H_{m,\ell}$ as 
\begin{equation}\label{eqn:Htildeexp}
2\pi i \sum_{M\in \Gamma_{\infty}\backslash\SL_2(\Z)} \frac{1}{\im(M\mathfrak{z})^{\ell}} \cdot \frac{\left(c\mathfrak{z}+d\right)^{-m}}{1- e^{2\pi i (z-M\mathfrak{z})}} =\frac{2\pi i}{\mathfrak{z}_2^{\ell}}\sum_{n=0}^{\infty}\sum_{M\in \Gamma_{\infty}\backslash\SL_2(\Z)} \frac{\left|c\mathfrak{z}+d\right|^{2\ell}}{\left(c\mathfrak{z}+d\right)^{m}} e^{-2\pi i n M\mathfrak{z}}q^n.
\end{equation}
Since $\lambda(c,d)=\left|c\mathfrak{z}+d\right|^2$, we see that each summand occuring in \eqref{eqn:Hexp} is precisely multiplied by $
(\lambda(c,d)/\mathfrak{z}_2)^{\ell}
$.  Since the only dependence on $z$ in \eqref{eqn:Htildeexp} comes from $q^n$, the dependence on $r$ is clear.
This completes the proof of (1).  \\
\noindent (2) We directly obtain the claim from the $\mathfrak{z}=\rho$ and $\mathfrak{z}=i$ cases of (1) combined with Proposition \ref{realFourier}.
\end{proof}

We also require the following useful lemma about the growth of $H_{m,\ell}^{(\Lpow)}(\mathfrak{z}, z)$ as $\mathfrak{z}\to i\infty$ and $\mathfrak{z} \to z$.
\begin{lemma}\label{lem:Hgengrowth}
\noindent

\noindent
\begin{enumerate}[leftmargin=*]
\item[\rm(1)]
For every $z\in \H$, the functions $H_{m,\ell}^{(\Lpow)}(\mathfrak{z},z)$ vanish as $\mathfrak{z}\to i\infty$.
\item[\rm(2)] The following limit exists: 
$$
\lim_{\mathfrak{z}\to z}\left( H_{m,\ell}(\mathfrak{z},z) - \frac{\Res_{\mathfrak{z}=z}\left(H_{m}(\mathfrak{z},z)\right)}{\mathfrak{z}_2^{\ell}} \left(\mathfrak{z}-z\right)^{-1}\right).
$$

\item[\rm(3)] If $z\in \H$ is not an elliptic fixed point, then the following limit exists: 
$$
\lim_{\mathfrak{z}\to z}\left( H_{m,\ell}^{(r)}(\mathfrak{z},z) - 2\frac{r!}{\mathfrak{z}_2^{\ell}} \left(\mathfrak{z}-z\right)^{-r-1}\right).
$$
\end{enumerate}
\end{lemma}
\begin{proof}
(1)   We first rewrite
\begin{equation}\label{eqn:Hboundexp}
H_{m,\ell}^{(r)}(\mathfrak{z},z) = \frac{4\pi i }{\mathfrak{z}_2^{\ell}}\left( -\frac{\partial^r}{\partial z^r}\frac{e^{2\pi i (\mathfrak{z}-z)}}{1-e^{2\pi i (\mathfrak{z}-z)}} + \sum_{\substack{M\in \Gamma_{\infty}\backslash \SL_2(\Z)\\ M\notin \pm \Gamma_{\infty}}} \frac{|c\mathfrak{z}+d|^{2\ell}}{\left(c\mathfrak{z}+d\right)^{m}}\frac{\partial^r}{\partial z^r}\left(\frac{1}{1-e^{2\pi i (z-M\mathfrak{z})}}\right)\right).
\end{equation}
Since $\mathfrak{z}\to i\infty$ and $z$ is fixed, we may assume that there exists a constant $K>1$ such that $ \frac{1}{K}<y<K\text{ and }\mathfrak{z}_2>K$.
We may then expand the first term in \eqref{eqn:Hboundexp} as 
$$
-\sum_{n=1}^{\infty}  \left(-2\pi i n\right)^r e^{2\pi i n(\mathfrak{z}-z)},
$$
which converges absolutely since $\mathfrak{z}_2-y>0$ and decays exponentially.

We finally bound the second term in \eqref{eqn:Hboundexp}.
For $M\notin \pm \Gamma_{\infty}$, we have 
$$
\im(M\mathfrak{z}) =\frac{\mathfrak{z}_2}{|c\mathfrak{z}+d|^2}<\frac{1}{c^2\mathfrak{z}_2}<\frac{1}{K}<y.
$$
Hence we may expand the geometric series in the second term of \eqref{eqn:Hboundexp} and estimate it against
\begin{align*}
&\ll \mathfrak{z}_2^{-\ell}\sum_{n=0}^{\infty} (
2\pi n
)^r e^{-2\pi n\left(y-\frac{1}{K}\right)}\sum_{\substack{(c,d)=1\\ c>0}} |c\mathfrak{z}+d|^{2\ell-m}\\
&=\mathfrak{z}_2^{-\frac{m}{2}}\sum_{n=0}^{\infty} (
2\pi n
)^r e^{-2\pi n\left(y-\frac{1}{K}\right)}\sum_{\substack{M\in \pm\Gamma_{\infty}\backslash \SL_2(\Z)\\ M\notin \pm \Gamma_{\infty}}} \im(M\mathfrak{z})^{\frac{m}{2}-\ell}.
\end{align*}
The first factor is bounded as $\mathfrak{z}_2\to \infty$, the second is independent of $\mathfrak{z}$, and the third decays exponentially because it is a weight zero nonholomorphic Eisenstein series without 
the term corresponding to the identity matrix.
  This completes the proof of (1).

\noindent
(2) The claim follows directly from the definition, since 
\begin{equation}\label{eqn:Hgrowth}
H_{m,\ell}(\mathfrak{z},z) = \frac{A_z}{\mathfrak{z}_2^{\ell}(\mathfrak{z}-z)} +O_{z}(1), \ \text{where}
\end{equation}
$$
A_z:=\Res_{\mathfrak{z}=z}H_{m}(\mathfrak{z},z).
$$
(3) Note that if $z$ is not an elliptic fixed point, then there exists a neighborhood around $z$ for which $A_z=2$ by \eqref{res}.  Therefore the dependence on $z$ in $O_z(1)$ of \eqref{eqn:Hgrowth} is holomorphic as a function of $z$ in this neighborhood.  It follows that
$$
H_{m,\ell}^{(\Lpow)}(\mathfrak{z},z) = 2\frac{r!}{\mathfrak{z}_2^{\ell}}(\mathfrak{z}-z)^{-r-1} +O_{z}(1), 
$$
completing the proof.
\end{proof}

As a special case of Lemma \ref{lem:Hgengrowth} (1), one concludes, for $\ell=0$, that $H_{m}^{(r)}$ is a meromorphic cusp form.  
We next show that these are the only such forms modulo cusp forms.
\begin{lemma}\label{lem:merospan}
If $z\in\H$ is not an elliptic fixed point, then the space of weight $m\in 2\N$ meromorphic cusp forms whose only poles modulo $\SL_2(\Z)$ are poles at $\mathfrak{z}=z$ of order at most $r\in\N_0$ is spanned by cusp forms and 
$$
\left\{ H_{m}^{(j)}(\mathfrak{z},z)\middle| 0\leq j\leq r-1\right\}.
$$
\end{lemma}
\begin{proof}
Since each function
$H_{m}(\mathfrak{z},z)$ is a meromorphic cusp form with a simple pole at $\mathfrak{z}=z$ whenever $z$ is not an elliptic fixed point, this follows immediately by taking derivatives.
\end{proof}

\subsection{Second-order poles}
We now specialize to meromorphic modular forms with second-order poles.  We obtain the Fourier expansion of $Y_{2-m,-2}\left(\mathfrak{z},z\right)$ by the argument suggested in Section \ref{HigherOrder}.  
\begin{proposition}\label{thm:Order2poles}
For every $\mathfrak{z}\in \{i,\rho\}$, $m>2/\omega_{\mathfrak{z}}$, and $y>\mathfrak{z}_2$, we have
$$
Y_{4-2m\omega_{\mathfrak{z}}, -2}(\mathfrak{z},z)=\frac{4\pi i \omega_{\mathfrak{z}}}{\left(2i\mathfrak{z}_2\right)^{1-2m\omega_{\mathfrak{z}}}} \sum_{n=0}^{\infty} \;\sideset{}{^*}\sum_{\mathfrak{b}\subseteq\mathcal{O}_{\Q(\mathfrak{z})}} \frac{C_{2m\omega_{\mathfrak{z}}}\left(\mathfrak{b},n\right)}{N(\mathfrak{b})^{m\omega_{\mathfrak{z}}}}\left(\frac{m\omega_{\mathfrak{z}}-1}{i\mathfrak{z}_2}N(\mathfrak{b})-2\pi i n\right)e^{\frac{2 \pi n\mathfrak{z}_2}{N(\mathfrak{b})}} q^n.
$$
\end{proposition}

Following Section \ref{HigherOrder}, to obtain the Fourier coefficients of meromorphic modular forms with second-order poles, it remains to write $Y_{2-m,-2}(\mathfrak{z},z)$ in terms of the functions $H_{m,\ell}^{(\Lpow)}(\mathfrak{z},z)$.  

\begin{proposition}\label{prop:Ycoeff}
We have 
\begin{equation*}\label{rewriteYm}
Y_{2-m,-2}(\mathfrak{z},z)=  \left(2i\mathfrak{z}_2\right)^{m+1}\left(\frac{m}{2i} H_{m+2,1}\left(\mathfrak{z},z\right) - H_{m+2}^{(1)}\left(\mathfrak{z},z\right)\right).
\end{equation*}
\end{proposition}
\begin{proof}
Using the chain rule and the product rule, we compute 
\begin{equation}\label{computeY}
Y_{2-m,-2}(\mathfrak{z}, z)=\left[\frac{\partial\alpha}{\partial{X_\mathfrak{z}}
(\alpha)}\left(m\left(\alpha-\overline{\mathfrak{z}}\right)^{m-1}H_{m}(\alpha, z)+\left(\alpha-\overline{\mathfrak{z}}\right)^{m}\frac{\partial}{\partial\alpha} H_{m}(\alpha, z)\right)\right]_{\alpha=\mathfrak{z}}.
\end{equation}
We directly evaluate 
\begin{equation}\label{eqn:wdiff}
\frac{\partial}{\partial\alpha}{X_\mathfrak{z}}
(\alpha)=\frac{2i \mathfrak{z}_2}{\left(\alpha-\overline{\mathfrak{z}}\right)^2}.
\end{equation}
To rewrite the second term  in (\ref{computeY}), we compare termwise the derivatives of the functions $H_{m}(\alpha,z)$ with respect to $\alpha$ 
and their derivatives with respect to $z$.  
For this, we compute
\begin{align}\label{computeterms}
\frac{\partial}{\partial\alpha}\left(\frac{(c\alpha+d)^{-m}}{1-e^{2\pi i(z-M\alpha)}}\right)
&=-\frac{mc(c\alpha+d)^{-m-1}}{1-e^{2\pi i(z-M\alpha)}}+\frac{(c\alpha+d)^{-m}e^{2\pi i(z-M\alpha)}}{\left(1-e^{2\pi i(z-M\alpha)}\right)^2}\frac{(-2\pi i)}{(c\alpha+d)^2},\\
\nonumber
\frac{\partial}{\partial z}\frac{(c\alpha+d)^{-m-2}}{1-e^{2\pi i(z-M\alpha)}}&=\frac{(c\alpha+d)^{-m-2}2\pi ie^{2\pi i(z-M\alpha)}}{\left(1-e^{2\pi i(z-M\alpha)}\right)^{2}}.
\end{align}
We combine the first term in (\ref{computeterms}) with the first term coming from \eqref{computeY}.  After plugging in each term has the shape $(2i\mathfrak{z}_2)^{m}$ times
\[
m\frac{(c\mathfrak{z}+d)^{-m-1}}{1-e^{2\pi i(z-M\mathfrak{z})}}\left(c\mathfrak{z}+d-c\left(\mathfrak{z}-\overline{\mathfrak{z}}\right)\right)
=\frac{m(c\mathfrak{z}+d)^{-m-2}\text{Im}(M\mathfrak{z})^{-1}}{1-e^{2\pi i(z-M\mathfrak{z})}}\mathfrak{z}_2.
\]
This term contributes 
\[
\frac{m}{2i}\left(2i\mathfrak{z}_2\right)^{m+1}H_{m+2,1}(\mathfrak{z}, z).
\]
The second term in the first identity of \eqref{computeterms} produces
\[
-\left(2i\mathfrak{z}_2\right)^{m+1} H_{m+2}^{
 (1)
 }(\mathfrak{z}, z).
\]
Plugging back into \eqref{computeY} yields the statement of the proposition.
\end{proof}

We now have all of the preliminaries necessary to prove Proposition \ref{thm:Order2poles}.

\begin{proof}[Proof of Proposition \ref{thm:Order2poles}]
The theorem follows immediately by plugging the Fourier expansions of $H_{m+2,1}(\mathfrak{z},z)$ and $H_{m+2}(\mathfrak{z},z)$ from Theorem \ref{genFC} into \eqref{rewriteYm}.
\end{proof}

\subsection{Third-order poles}

Analogously to Proposition \ref{thm:Order2poles}, we prove the following Fourier expansions.
\begin{proposition}\label{thm:Order3poles}
For every $\mathfrak{z}\in \{i,\rho\}$, $m>4/\omega_{\mathfrak{z}}$ and $y>\mathfrak{z}_2$, one has 
\begin{multline*}
Y_{6-2m\omega_{\mathfrak{z}},-3}(\rho,z) = 2\pi i \omega_{\mathfrak{z}} \left(2i\mathfrak{z}_2\right)^{2m\omega_{\mathfrak{z}}-2}\sum_{n=0}^{\infty} \; \sideset{}{^*}\sum_{\mathfrak{b}\subseteq\mathcal{O}_{\Q(\mathfrak{z})}} \frac{C_{2m\omega_{\mathfrak{z}}}\left(\mathfrak{b},n\right)}{N(\mathfrak{b})^{m\omega_{\mathfrak{z}}}}\\
\times \left(-\frac{1}{3}(2m\omega_{\mathfrak{z}}-4)(2m\omega_{\mathfrak{z}}-3)N(\mathfrak{b})^2-\frac{2\pi n}{\mathfrak{z}_2}\left(2m\omega_{\mathfrak{z}}-3\right)N(\mathfrak{b}) - 4\pi^2n^2\right)e^{\frac{2\pi n\mathfrak{z}_2}{N(\mathfrak{b})}} q^n.
\end{multline*}
\end{proposition}
In order to prove Proposition \ref{thm:Order3poles},
 we again rewrite $Y_{2-m,-3}(\mathfrak{z},z)$ in terms of the functions $H_{m,\ell}^{(\Lpow)}(\mathfrak{z},z)$.  

\begin{proposition}\label{prop:Y3coeff}
We have 
$$
Y_{2-m,-3}(\mathfrak{z},z)=\frac{1}{2}\left(2i\mathfrak{z}_2\right)^{m+2}\Bigg( -\frac{m(m+1)}{4}H_{m+4,2}(\mathfrak{z},z)
+i(m+1)H_{m+4,1}^{(1)} \left(\mathfrak{z},z\right)
+H_{m+4}^{(2)}\left(\mathfrak{z},z\right)\Bigg).
$$
\end{proposition}
\begin{proof}
The claim follows by a long and tedious calculation similar to the proof of Proposition \ref{prop:Ycoeff}.  We point out a few key features.  Whenever one takes the derivative with respect to $\mathfrak{z}$ of $c\mathfrak{z}+d$, one combines the resulting $c$ with $2i\mathfrak{z}_2=\mathfrak{z}-\overline{\mathfrak{z}}$ to obtain 
$$
c\left(\mathfrak{z}-\overline{\mathfrak{z}}\right)= c\mathfrak{z}+d - \left(c\overline{\mathfrak{z}}+d\right).
$$
This then results in a sum of terms of the shape 
(for some $m_1\in\N$ depending on the term)
$$
\left(c\mathfrak{z}+d\right)^{-m_1}\left(c\overline{\mathfrak{z}}+d\right)^{\ell} = \left|c\mathfrak{z}+d\right|^{2\ell} \left(c\mathfrak{z}+d\right)^{-m_1-\ell}.
$$
In each case, $m_1+\ell$ turns out to be $m+4$ and the Poincar\'e series $H_{m+4,\ell}(\mathfrak{z},z)$ then 
appear by rewriting
$$
\left|c\mathfrak{z}+d\right|^{2\ell}=\im(M\mathfrak{z})^{-\ell}\mathfrak{z}_2^{\ell}.
$$
After this, the calculation reduces to simplification of the arising terms.

\end{proof}

\subsection{Proof of Theorem \ref{thm:mainmero}}
We 
now have
 all of the ingredients needed to prove Theorem \ref{thm:mainmero}.
\begin{proof}[Proof of Theorem \ref{thm:mainmero}]
By Satz 3 of \cite{Pe1}, for $\nu\in \N$, every 
$F\in \mathcal{M}_{2-m,\mathfrak{z},\nu}$
may be written as a linear combination of the functions $H_{m}(\mathfrak{z},z)$ and $Y_{2-m,-\mu}(\mathfrak{z},z)$ with $2\leq \mu\leq \nu$. The main statement now follows by Propositions
 \ref{realFourier}, \ref{thm:Order2poles}, and \ref{thm:Order3poles}.
\end{proof}

\section{Meromorphic quasi-modular forms and the Proof of Theorem \ref{thm:mainE2}}\label{sec:quasi}
In this section, we investigate meromorphic quasi-modular forms 
that are powers of $E_2$ times meromorphic modular forms and prove Theorem \ref{thm:mainE2}.  We determine the Fourier coefficients of such forms by relating them to the functions $H_{m,\ell}^{(\Lpow)}
(\mathfrak{z},z)
$. 
\subsection{A general construction}\label{sec:E2gen}

In order to determine the Fourier coefficients of meromorphic quasi-modular forms, we first require certain relations between the functions $\mathfrak{z}\mapsto H_{m,\ell}^{(r)}(\mathfrak{z},z)$ and meromorphic cusp forms.  To state the result, we first define the weight $2$ modular completion of $E_2$ by  
\begin{equation}\label{E2comp}
\widehat{E}_2(z):=E_2(z)-\frac3{\pi y}.
\end{equation}

We require the following elementary and well-known properties of $E_2$. 

\begin{lemma}\label{E2lemma}
\noindent

\noindent
\begin{enumerate}[leftmargin=*]
\item[\rm(1)] The function $\widehat{E}_2$ vanishes at $i$ and $\rho$.
\item[\rm(2)] We have
\[
L\left(\widehat{E}_2\right)=\frac{3}{\pi}.
\]
\end{enumerate}
\end{lemma}

Next define
\[
\mathcal{F}_{m,\ell}^{(r)}(\mathfrak{z}, z):=\sum_{j=0}^\ell\binom{\ell}{j}\left(\frac{3}{\pi}\right)^{\ell-j}\widehat{E}_2^j(\mathfrak{z})H_{m-2j, \ell-j}^{(r)}(\mathfrak{z}, z).
\]
If $r=0$, we omit the dependence on $r$.
\begin{lemma}\label{Fmodular}
If $m\geq 4+2\ell$, then the functions $\mathfrak{z}\mapsto \mathcal{F}_{m,\ell}^{(r)}(\mathfrak{z}, z)$ are meromorphic cusp forms of weight $m$.
\end{lemma}

\begin{proof}
Firstly, $\mathfrak{z}\mapsto \mathcal{F}_{m,\ell}^{(r)}(\mathfrak{z},z)$ satisfies weight $m$ modularity by construction.  We next show that it is meromorphic.  For this, we use 
the fact that
\rm
 $j\binom{\ell}{j}=\binom{\ell}{j-1}(\ell-j+1)$ and apply the lowering operator. 
By 
\rm
 Lemma \ref{E2lemma} and \eqref{lowerH}, one obtains $L_{\mathfrak{z}}(\mathcal{F}_{m,\ell}^{(r)}(\mathfrak{z}, z))=0$ and thus the meromorphicity. To finish the proof, we are left to prove that $\mathcal{F}_{m,\ell}^{(r)}(\mathfrak{z}, z)$ vanishes as $\mathfrak{z}\to i\infty$.  
This follows by Lemma \ref{lem:Hgengrowth} (1) and the fact that $\widehat{E}_2$ is bounded as $\mathfrak{z}\to i\infty$.
\end{proof}

The following lemma is useful to determine the principal parts of $\mathfrak{z}\mapsto \mathcal{F}_{m,\ell}^{(r)}(\mathfrak{z},z)$.
\begin{lemma}\label{lem:Fmellgrowth}
If $m\geq 4+2\ell$ and $z$ is not an elliptic fixed point, then the limit 
$$
\lim_{\mathfrak{z}\to z}\left(\mathcal{F}_{m,\ell}^{(r)}(\mathfrak{z},z)-2E_2^{\ell}(\mathfrak{z})(\mathfrak{z}-z)^{-r-1}\right)
$$
exists.
\end{lemma}
\begin{proof}
By Lemma \ref{lem:Hgengrowth} (3), we have 
\begin{align*}
\mathcal{F}_{m,\ell}^{(r)}(\mathfrak{z},z) &=2\sum_{j=0}^{\ell} \binom{\ell}{j} \left(\frac{3}{\pi \mathfrak{z}_2}\right)^{\ell-j} \widehat{E}_2^j(\mathfrak{z}) \left(\mathfrak{z}-z\right)^{-r-1}+ O(1)\\
&= 2\left(\frac{3}{\pi \mathfrak{z}_2}+ \widehat{E}_2(\mathfrak{z})\right)^\ell(\mathfrak{z}-z)^{-r-1}+O(1) = 2E_2^{\ell}(\mathfrak{z})(\mathfrak{z}-z)^{-r-1} + O(1).
\end{align*}
This concludes the proof.
\end{proof}

Lemmas \ref{Fmodular} and \ref{lem:Fmellgrowth} yield another useful representation of $\mathcal{F}_{m,\ell}^{(r)}(\mathfrak{z},z)$ as linear combinations of meromorphic cusp forms $H_{m}^{(j)}(\mathfrak{z},z)$ with $0\leq j\leq r$, where the coefficients of the linear combination are products of powers of $E_2(z)$ and derivatives of $E_2(z)$.
\begin{proposition}\label{thm:Htildegen}
If $m\geq 4+2\ell$, $\Lpow\in\N_0$, and $z$ is not an elliptic fixed point, then
$$
\mathfrak{z}\mapsto\mathcal{F}_{m,\ell}^{(r)}(\mathfrak{z},z) -\frac{1}{\Lpow!}\frac{\partial^{\Lpow}}{\partial z^{\Lpow}}\Big( E_2^{\ell}(z)H_m(\mathfrak{z},z)\Big)
$$
are cusp forms of weight $m$.
\end{proposition}
\begin{proof}
By Lemma \ref{Fmodular}, we have that the functions $\mathfrak{z}\mapsto \mathcal{F}_{m,\ell}^{(r)}(\mathfrak{z},z)$ are meromorphic cusp forms of weight $m$ 
and Lemma \ref{lem:Fmellgrowth} implies that the pole at $\mathfrak{z}=z$ of each function has order at most $r+1$.
  By Lemma \ref{lem:merospan}, this space is spanned by cusp forms and $\{ H_{m}^{(j)}(\mathfrak{z},z)| 0\leq j\leq r\}$.
 To determine the explicit coefficients $C_{\ell,0}(z),\dots, C_{\ell,r}(z)$ so that 
$$
\mathcal{F}_{m,\ell}^{(r)}(\mathfrak{z},z)-\sum_{j=0}^{r} C_{\ell,j}(z)H_{m}^{(j)}(\mathfrak{z},z)
$$
are cusp forms, we plug in Lemma \ref{lem:Fmellgrowth} 
and the Taylor expansion of $E_2^{\ell}$
around $\mathfrak{z}=z$.
  This yields that
\begin{equation}\label{eqn:Fgenexp}
\mathcal{F}_{m,\ell}^{(r)}(\mathfrak{z},z)=2\sum_{j=0}^{r}\frac{1}{j!}
\frac{\partial^j}{\partial z^j}
\Big(E_2^{\ell}(z)\Big)   \left(\mathfrak{z}-z\right)^{-r+j-1}+ O(1).
\end{equation}
Moreover, by
 Lemma \ref{lem:Hgengrowth} (3), we have 
$$
H_{m}^{(r-j)}(\mathfrak{z},z) = 2\left(r-j\right)!(\mathfrak{z}-z)^{-r+j-1} + O(1).
$$
Hence 
$$
C_{\ell,j}(z) = \frac{1}{j!(r-j)!}
\frac{\partial^j}{\partial z^j}
 \Big(E_2^{\ell}(z)\Big),
$$
which yields the statement of the theorem after using the product rule.  
\end{proof}

\subsection{Proof of Theorem \ref{thm:mainE2}}
In this subsection, we investigate powers of $E_2$ multiplied by meromorphic modular forms with at most second-order poles 
 to prove Theorem \ref{thm:mainE2}.  For small weights, we require a slightly more explicit version of 
Proposition \ref{thm:Htildegen},
 given in the following corollary.
\begin{corollary}\label{cor:Htildegen}
\noindent

\noindent
\begin{enumerate}[leftmargin=*]
\item[\rm(1)]
If
 $m\in \left\{6,8,10,14\right\}$, then 
$$
H_{m,1}\left(\mathfrak{z},z\right)=\frac{\pi}{3}E_2(z)H_{m}\left(\mathfrak{z},z\right) -\frac{\pi}{3}\widehat{E}_2(\mathfrak{z})H_{m-2}\left(\mathfrak{z},z\right).
$$
Moreover, if $m\in \left\{8,10,14\right\}$, then 
$$
H_{m,2}\left(\mathfrak{z},z\right)=\left(\frac{\pi}{3}\right)^2E_2^2(z)H_{m}\left(\mathfrak{z},z\right)-\frac{2\pi}{3}\widehat{E}_2(\mathfrak{z})H_{m-2,1}(\mathfrak{z},z) -\left(\frac{\pi}{3}\right)^2\widehat{E}_2^2(\mathfrak{z})H_{m-4}\left(\mathfrak{z},z\right).
$$
\item[\rm(2)]
If $m\in\{ 6,8, 10,14\}$, then for all $z,\mathfrak{z}\in \H$ we have 
$$
E_2(z)H_{m}^{(1)}(\mathfrak{z},z)=\frac{3}{\pi} H_{m,1}^{(1)}(\mathfrak{z},z)
 + \widehat{E}_2(\mathfrak{z}) H_{m-2}^{(1)}(\mathfrak{z},z)- \left(\frac{\pi i }{6}\left( E_2^2(z)-E_4(z)\right)\right)H_{m}(\mathfrak{z},z).
$$
\end{enumerate}
\end{corollary}

In order to conclude Corollary \ref{cor:Htildegen} from 
Proposition \ref{thm:Htildegen},
 we require the principal parts of $\mathcal{F}_{m,\ell}^{(r)}(\mathfrak{z},z)$.  By Lemma \ref{lem:Fmellgrowth}, these are
 determined by  the first $r$ Taylor coefficients of $E_2^{\ell}$ around a point $z\in\H$.  For our purposes, the first two coefficients suffice.
\begin{lemma}\label{lem:E2exp}
In a neighborhood of $z$ we have for every $\ell\in\N$
$$
E_2^{\ell}(\mathfrak{z}) = E_2^{\ell}(z) +\frac{\pi i \ell }{6}\left(E_2^{\ell+1}(z) - E_2^{\ell-1}(z)E_4(z)\right)\left(\mathfrak{z}-z\right) + O\Big(\left(\mathfrak{z}-z\right)^2\Big).
$$
\end{lemma}
\begin{proof}
The claim follows 
from
 (see (30) of \cite{Ra})
\begin{equation}\label{eqn:E2deriv}
E_2'(\mathfrak{z}) = \frac{\pi i}{6}\left(
E_2^2
(\mathfrak{z})-E_4(\mathfrak{z})\right).
\end{equation}
\end{proof}

\begin{proof}[Proof of Corollary \ref{cor:Htildegen}]
\noindent
(1)
We first note that there are no cusp forms of weight $m$.  Hence if $z$ is not an elliptic fixed point, then the first (resp. second) identity holds by 
Proposition \ref{thm:Htildegen}
 with $\ell=1$ (resp. $\ell=2$) and $r=0$.  We 
then analytically continue in $z$ to obtain the claim for all $z$.

\noindent
(2)  We first explicitly plug in the Taylor coefficients of $E_2^{\ell}$ 
around $\mathfrak{z}=z$
 from Lemma \ref{lem:E2exp} to see that the claim holds 
by 
the $\ell=1$ and $r=1$ case of 
Proposition \ref{thm:Htildegen} 
whenever $z$ is not an elliptic fixed point.  The claim then follows by analytic continuation in $z$.
\end{proof}

We are now ready to prove Theorem \ref{thm:mainE2}.
\begin{proof}[Proof of Theorem \ref{thm:mainE2}]
First, assume that $F\in \mathcal{M}_{\mathfrak{z},1}$ with $\mathfrak{z}=i$ or $\mathfrak{z}=\rho$.  By Satz 1 of \cite{Pe1}, $F$ is a constant multiple of $H_{m}(\mathfrak{z},z)$ for some $m\in \N$.  Using 
Proposition \ref{thm:Htildegen}
with $r=0$ and the fact that the space of weight $m$ cusp forms is trivial, we obtain that 
$$
E_2^{\ell}(z)H_{m}(\mathfrak{z},z) = \mathcal{F}_{m,\ell}(\mathfrak{z},z).
$$
Furthermore, for $\mathfrak{z}\in \{i,\rho\}$, Lemma \ref{E2lemma} implies that 
$$
\mathcal{F}_{m,\ell}(\mathfrak{z},z) = \left(\frac{3}{\pi}\right)^{\ell}H_{m,\ell}(\mathfrak{z},z).
$$
We conclude that 
\begin{equation}\label{eqn:E2simple}
E_2^{\ell}(z)H_{m}(\mathfrak{z},z)= \left(\frac{3}{\pi}\right)^{\ell}H_{m,\ell}(\mathfrak{z},z).
\end{equation}
The Fourier expansions of $H_{m,\ell}(\mathfrak{z},z)$ have the desired form by Theorem \ref{genFC}.  This completes the proof of the theorem if the poles are simple.  

We next consider the case that $F\in \mathcal{M}_{\mathfrak{z},2}$ with $\mathfrak{z}\in\{i,\rho\}$.  Then $F$ is a linear combination of $H_{m}(\mathfrak{z},z)$ and $Y_{2-m,-2}(\mathfrak{z},z)$ by Satz 3 of \cite{Pe1}.  By Proposition \ref{prop:Ycoeff}, the functions $Y_{2-m,-2}(\mathfrak{z},z)$ are linear combinations of $H_{m+2,1}(\mathfrak{z},z)$ and $H_{m+2}^{(1)}(\mathfrak{z},z)$, and it hence suffices to prove that $E_2^{\ell}(z)H_{m+2}^{(1)}(\mathfrak{z},z)$ and $E_2^{\ell}(z)H_{m+2}^{(1)}(\mathfrak{z},z)$ both have Fourier expansions which may be written as linear combinations of the series $F_{m,\ell,r}(\mathfrak{z};q)$ given in \eqref{Fklr}.  
However, by \eqref{eqn:E2simple}, we have that
$$
E_2^{\ell}(z) H_{m+2,1}(\mathfrak{z},z)= \frac{\pi}{3}E_2^{\ell+1}(z)H_{m+2}(\mathfrak{z},z),
$$
for which we have already shown the claim (since the space of weight $m+2$ cusp forms is trivial).   Since there are no cusp forms of weight $m+2$, the $r=1$ case of 
Proposition \ref{thm:Htildegen} together with Lemma \ref{lem:E2exp} (noting that the coefficient of $(\mathfrak{z}-z)$ in the Taylor expansion at $\mathfrak{z}=z$ of $E_{2}^{\ell}(\mathfrak{z})$ is $\frac{\partial}{\partial z}(E_2^{\ell}(z))$) implies that
\begin{align*}
E_2^{\ell}(z)H_{m+2}^{(1)}(\mathfrak{z},z) &= \mathcal{F}_{m+2,\ell}^{(1)}(\mathfrak{z},z) - \frac{\partial}{\partial z}\left(E_2^{\ell}(z)\right)H_{m+2}(\mathfrak{z},z)\\
&= \mathcal{F}_{m+2,\ell}^{(1)}(\mathfrak{z},z) - E_2^{\ell-1}(z)\left(\frac{\pi i \ell}{6} \left(E_2^2(z) - E_4(z)\right)\right)H_{m+2}(\mathfrak{z},z).
\end{align*}
However, Lemma \ref{E2lemma} (1) implies that 
$$
 \mathcal{F}_{m+2,\ell}^{(1)}(\mathfrak{z},z) =\left(\frac{3}{\pi}\right)^{\ell} H_{m+2,\ell}^{(1)}(\mathfrak{z},z).
$$
We conclude that 
$$
E_2^{\ell}(z)H_{m+2}^{(1)}(\mathfrak{z},z)= \left(\frac{3}{\pi}\right)^{\ell}H_{m+2,\ell}^{(1)}(\mathfrak{z},z) - E_2^{\ell-1}(z)\left(\frac{\pi i \ell}{6} \left(E_2^2(z) - E_4(z)\right)\right)H_{m+2}(\mathfrak{z},z).
$$

By Theorem \ref{genFC}, $H_{m+2,\ell}^{(1)}(\mathfrak{z},z)$ have Fourier expansions of the desired type.  Since the space of weight $m+2$ cusp forms is trivial, the space of weight $m-2$ cusp forms is also trivial, and we conclude that 
 $H_{m+2}(\mathfrak{z},z)\in\mathcal{M}_{-m,\mathfrak{z},1}$ and $E_4(z)H_{m+2}(\mathfrak{z},z)\in \mathcal{M}_{4-m,\mathfrak{z},1}$.  Using Theorem \ref{thm:mainE2} in the case when the meromorphic modular forms have simple poles (shown above) implies that powers of $E_2$ times  these functions have Fourier expansions of the desired shape, finishing the proof.
\end{proof}

 \section{Examples and the proofs of Corollaries \ref{cor:explicit} and \ref{cor:E2explicit}}\label{sec:examples}
In order to prove Corollary \ref{cor:explicit} and Corollary \ref{cor:E2explicit}, we must explicitly write a number of meromorphic modular forms in terms of the functions $H_{m}(\mathfrak{z},z)$ and $Y_{2-m,\nu}(\mathfrak{z},z)$.
\subsection{Simple poles}
  We first write the functions of interest in terms of meromorphic Poincar\'e series.
 \begin{lemma}\label{residue simple}
 We have
 \begin{align*}
 \frac1{E_4(z)} &=\frac1{4\pi iE_6(\rho)} H_6(\rho, z),\\
  \frac1{E_6(z)} &=\frac1{4\pi iE_4^2(i)} H_8(i, z),\\
 \frac{E_4(z)}{E_6(z)} &=\frac1{4\pi iE_4(i)} H_4(i, z).
 \end{align*}
 \end{lemma}
\begin{remark}
Using the Chowla-Selberg formula (cf. the corollary to Proposition 27 of \cite{123}), the constant $E_4(i)$ may be rewritten as
$$
E_4(i) = \frac{3 \Gamma\left(\frac{1}{4}\right)^8}{(2\pi)^6}.
$$
\end{remark}
\begin{proof}
Noting that the spaces of weight $6$, $8$, and $4$ cusp forms, 
respectively, are all trivial, Satz 1 of \cite{Pe1} implies that the right-hand sides of each of the above identities are meromorphic modular forms.  Since both sides of the identities are meromorphic modular forms of the same (negative) weight with poles at the same point it suffices to prove that their principal parts agree.
 Firstly, by \eqref{eqn:zRes}, we have
 \[
 \Res_{z=\rho} H_6(\rho, z)=-6.
 \]
 Moreover
 \[
 \Res_{z=\rho}\frac1{E_4(z)}=\frac1{E_4'(\rho)}.
 \]
 Now the first identity follows from 
(see (30) of \cite{Ra})
 \begin{equation}\label{E4D}
 \frac1{2\pi i} E_4' (z)=\frac13\left(E_2(z) E_4(z)-E_6(z)\right).
 \end{equation}
 The remaining identities can be concluded similarly, using that 
(see (30) of \cite{Ra})
 \begin{equation}\label{E6D}
  \frac1{2\pi i} E_6' (z)=\frac12\left(E_2(z) E_6(z)-E_4^2(z)\right).
 \end{equation}
 \end{proof}
 We next consider multiplication by powers of $E_2$. 
 
 \begin{lemma}\label{residuesimpleE2}
 We have
 \begin{align*}
 \frac{E_2(z)}{E_6(z)} &=\frac{3}{4\pi^2 iE_4^2(i)} H_{8,1}(i, z),\\
 \frac{E_2(z)}{E_4(z)} &=\frac{3}{4\pi^2 iE_6(\rho)} H_{6,1}(\rho, z),\\
 \frac{E_2^2(z)}{E_6(z)} &=\frac{9}{4\pi^3 i E_4^2(i)} H_{8,2}(i, z).
 \end{align*}
 \end{lemma}

 \begin{proof}
By Lemma \ref{residue simple} we have
\[
 \frac{E_2(z)}{E_6(z)}=\frac1{4\pi iE_4^2(i)} H_8(i, z) E_2(z). 
\]
Now 
Corollary \ref{cor:Htildegen} (1)
 concludes the first identity. The second identity is proven in the same way.  The third claim follows by  
Corollary \ref{cor:Htildegen} (1)
 and Lemma \ref{E2lemma} (1).

\end{proof}

 \subsection{Second-order poles}
We next write certain meromorphic modular forms with second-order poles in terms of 
the functions $Y_{2-m,-2}$.
\begin{lemma}\label{secondpolelemma}
 We have  
\begin{align*}
\frac{1}{E_4^2(z)} & = \frac{i}{648 \sqrt{3} \pi^2 E_6^2(\rho)}Y_{-8,-2}(\rho,z), \\
\frac{1}{E_6^2(z)} & = \frac{i}{2^{17} \pi^2 E_4^4(i)}Y_{-12,-2}(i,z),\\
\frac{E_4(z)}{E_6^2(z)} & = \frac{i}{2^{13} \pi^2  E_4^3(i)}Y_{-8,-2}(i,z), \\
\frac{E_4^2(z)}{E_6^2(z)} & = \frac{i}{2^{9} \pi^2 E_4^2(i)}Y_{-4,-2}(i,z),\\
\frac{E_6(z)}{E_4^2(z)} & = -\frac{i}{24 \sqrt{3} \pi^2 E_6 (\rho)}Y_{-2,-2}(\rho,z).
\end{align*}
\end{lemma}
 \begin{proof}
  We only prove the first identity, since the argument for the other ones is entirely analogous.  The left-hand side of the identity is clearly a meromorphic modular form of weight $-8$ which has a second-order pole at $z= \rho$.  By Satz 3 of \cite{Pe1} and the fact that the space of weight $10$ cusp forms is trivial, the right-hand side is also a meromorphic modular form of weight $-8$ with at most a second-order pole at $z=\rho$.  To finish the proof, we have to show that the principal parts around $\rho$ agree. For this, we use that (see (4a.9) of \cite{Pe1})
$$
Y_{2-m,\nu}(\rho,z) =-6\sqrt{3}i\, \delta\!\left(\frac{-\nu-1 + \frac{m}{2}}{3}\right) \left(z-\overline{\rho}\right)^{m-2} X_{\rho}^{\nu}(z) + \Phi_{\nu}(z),
$$ 
where $\Phi_{\nu}$ is regular in a neighborhood of $\rho$ and $\delta(x)=1$ if $x\in \Z$ and $\delta(x)=0$ otherwise.  
In particular, this yields that 
\begin{equation}\label{eqn:Y10exp}
Y_{-8,-2}(\rho, z) = - 6\sqrt{3} i\frac{\left(z-\overline{\rho}\right)^{10}}{(z-\rho)^2} + O(1).
\end{equation}
Furthermore, by (2a.16) of \cite{Pe1}, every weight $2-m$ meromorphic modular form has an expansion in a punctured annulus around $z_0$ of the type 
\begin{equation}\label{eqn:ellipticexp}
f(z)=\left(z-\overline{z_0}\right)^{m-2}\sum_{\substack{n\gg -\infty\\ n\equiv \frac{m}{2}-1\pmod{\omega_{z_0}}}} b_n\left(z_0,f\right) \left(\frac{z-z_0}{z-\overline{z_0}}\right)^n.
\end{equation}
  Since $1/E_4^2$ has at most a second-order pole at $\rho$, the congruence conditions in the expansion imply that
$$
\frac{1}{E_4^2(z)} = b_{-2}\left(\rho,\frac{1}{E_4^2}\right) \frac{\left(z-\overline{\rho}\right)^{10}}{\left(z-\rho\right)^{2}} + O(1).
$$
In order to compare $b_{-2}(\rho,1/E_4^2)$ with the constant in \eqref{eqn:Y10exp}, we evaluate the limit
$$
\lim_{z\to \rho} \frac{\left(z-\rho\right)^2}{E_4^2(z)} =  -\frac{9}{4 \pi^2 E_6^2(\rho)}.
$$
However, \eqref{eqn:Y10exp} implies that
$$
\lim_{z\to \rho} \left(z-\rho\right)^2 Y_{-8,-2}(\rho, z) = 2\cdot 3^{6} \sqrt{3}i.
$$
This concludes the proof.
 \end{proof}

We now use Lemma \ref{secondpolelemma} and the results from Section \ref{sec:quasi} to obtain $E_2\cdot 
E_4^2/E_6^2
$ as an explicit linear combination of the functions $H_{m,\ell}^{(r)}(i,z)$.
\begin{lemma}\label{lem:E2E4^2/E6^2}
 We have 
$$
\frac{E_2(z)E_4^2(z)}{E_6^2(z)}=\frac{1}{4 \pi^2 E_4^2(i)}\left(-\frac{3}{\pi}H_{8,1}^{(1)}(i,z) - \frac{\pi i E_4(i)}{6}H_4(i,z) - \frac{5 i}{2} \frac{3}{\pi} H_{8,2}(i,z)\right).
$$
\end{lemma}
\begin{proof}
Lemma \ref{secondpolelemma} gives that 
$$
\frac{E_2(z)E_4^2(z)}{E_6^2(z)}= \frac{i}{2^9 \pi^2 E_4^2(i)}Y_{-4,-2}(i,z)E_2(z).
$$
By 
Proposition \ref{prop:Ycoeff}, \eqref{eqn:E2simple},
Corollary \ref{cor:Htildegen} (1), and Corollary \ref{cor:Htildegen} (2),
we have that
$$
Y_{-4,-2}(i,z)E_2(z)= -i 2^7 \left(-\frac{3}{\pi}H_{8,1}^{(1)}(i,z) - \frac{\pi i}{6}E_4(z)H_8(i,z) - \frac{5i}{2}\cdot\frac{3}{\pi} H_{8,2}(i,z)\right).
$$
One then concludes the proof by applying Lemma \ref{residue simple} twice to obtain 
$$
E_4(z)H_8(i,z) = E_4(i)H_4(i,z).
$$
\end{proof}
\subsection{Third-order poles}
Similarly to the proof of Lemma \ref{secondpolelemma}, we relate meromorphic modular forms with third-order poles to 
the functions $Y_{2-m,-3}$.

\begin{lemma}\label{thirdpolelemma}
 We have
\begin{align*}
 \frac{1}{E_4^3(z)}&= \frac{i}{16 \cdot 3^6 \pi^3 E_6^3(\rho)}Y_{-12,-3}(\rho,z),\\
 \frac{E_6(z)}{E_4^3(z)}&= -\frac{i}{16 \cdot 3^3 \pi^3 E_6^2(\rho)}Y_{-6,-3}(\rho,z).
\end{align*}
\end{lemma}

\subsection{Proof of Corollaries \ref{cor:explicit} and \ref{cor:E2explicit}}
We are now ready to prove Corollary \ref{cor:explicit}.

\begin{proof}[Proof of Corollary \ref{cor:explicit}]
We obtain (1) by plugging 
Proposition \ref{thm:Order2poles}
 into Lemma \ref{secondpolelemma}. Part (2) follows by 
Proposition \ref{thm:Order3poles}
 and the identities from Lemma \ref{thirdpolelemma}.
In particular, the constants $k_f$, $a_f$, and $c_f$ in Corollary \ref{cor:explicit} (1) are:
\begin{equation}\label{eqn:explicitconsts1}
\begin{array}{|l|c|c|c|}
\hline
f& k_f&a_f&c_f\\
\hline
\hline
\frac{1}{E_4^2} & 12& \frac{15\sqrt{3}^{\vphantom{\frac{1}{2}}}}{\pi E_6^2(\rho)} & \frac{9}{E_6^2(\rho)}\\[1.2ex]
\frac{E_6}{E_4^2} & 12 & \frac{6\sqrt{3}}{\pi E_6(\rho)} & \frac{9}{E_6(\rho)}\\[1.2ex]
\frac{1}{E_6^2} & 16 & \frac{14}{\pi E_4^4(i)} & \frac{4}{E_4^4(i)}\\[1.2ex]
\frac{E_4}{E_6^2} & 12 & \frac{10}{\pi E_4^3(i)} & \frac{4}{E_4^3(i)}\\[1.2ex]
\hline
\end{array}
\end{equation}
The constants $k_f$, $a_f$, $c_f$, and $d_f$ in Corollary \ref{cor:explicit} (2) are:
\begin{equation}\label{eqn:explicitconsts2}
\begin{array}{|l|c|c|c|c|}
\hline
f& k_f&a_f&c_f&d_f\\
\hline
\hline
\frac{1}{E_4^3} & 18& \frac{945}{4\pi^2 E_6^3(\rho)} & \frac{135\sqrt{3}^{\vphantom{\frac{1}{2}}}}{2\pi E_6^3(\rho)}& \frac{27}{2 E_6^3(\rho)}\\[1.2ex]
\frac{E_6}{E_4^3} & 12 & \frac{81}{\pi^2 E_6^2(\rho)} &\frac{81\sqrt{3}}{2\pi E_6^2(\rho)} & \frac{27}{2E_6^2(\rho)}\\[1.2ex]
\hline
\end{array}
\end{equation}
\end{proof}

\begin{proof}[Proof of Corollary \ref{cor:E2explicit}]
By Lemma \ref{lem:E2E4^2/E6^2} and Theorem \ref{genFC}, we have 
\begin{equation}\label{eqn:explicitconsts3}
k=8,\quad a=\frac{12}{\pi E_4^2(i)},\quad c=-\frac{60}{E_4^2(i)},\text{ and } d=\frac{1}{3E_4(i)}.\end{equation}
\end{proof}

\section{An infinite family of meromorphic modular forms and the proof of Theorem \ref{thm:PoleArbitrary}}\label{sec:InfFamily}

In this section, we compute the Fourier expansions of $F_{\tau_0}$, defined in \eqref{eqn:Ftaudef}.  
\begin{proof}[Proof of Theorem \ref{thm:PoleArbitrary}]
We begin by rewriting $F_{\tau_0}$ as linear combinations of the functions $H_m(\mathfrak{z},z)$ and $Y_{2-m,\nu}(\mathfrak{z},z)$.  It is easy to see that $F_{\tau_0}$ are weight $-8$ meromorphic cusp forms with at most second-order poles at $\tau_0$ and no other poles in $\H$ modulo $\SL_2(\Z)$.  By Satz 3 of \cite{Pe1}, we may hence rewrite each  $F_{\tau_0}$ as a linear combination of $Y_{-8,-2}(\tau_0,z)$ and $H_{10}\left(\tau_0,z\right)$.  In order to determine the explicit linear combination, we must determine the principal parts of $F_{\tau_0}$.  

In order to find the principal parts of $F_{\tau_0}$, we first note that a direct calculation yields the Taylor expansion
$$
\left(j(z)-j\left(\tau_0\right)\right)^2=j'\left(\tau_0\right)^2 \left(z-\tau_0\right)^2 + j'\left(\tau_0\right)j''\left(\tau_0\right)\left(z-\tau_0\right)^3 +O\left(\left(z-\tau_0\right)^4\right).
$$
If $j'\left(\tau_0\right)\neq 0$, then it follows that
$$
\frac{1}{\left(j(z)-j\left(\tau_0\right)\right)^2} = j'\left(\tau_0\right)^{-2} \left(z-\tau_0\right)^{-2} - j'\left(\tau_0\right)^{-3}j''\left(\tau_0\right)\left(z-\tau_0\right)^{-1} + O(1).
$$
However, since 
$$
j'\left(\tau_0\right) =-2\pi i \frac{E_4^2\left(\tau_0\right)E_6\left(\tau_0\right)}{\Delta\left(\tau_0\right)},
$$
we have 
 $j'\left(\tau_0\right)=0$ if and only if $\tau_0=i$ or $\tau_0=\rho$ modulo $\SL_2(\Z)$.

We next expand 
$$
\frac{E_4(z)}{\Delta(z)} = \frac{E_4\left(\tau_0\right)}{\Delta\left(\tau_0\right)} + \frac{\Delta\left(\tau_0\right)E_4'\left(\tau_0\right) - E_4\left(\tau_0\right)\Delta'\left(\tau_0\right)}{\Delta\left(\tau_0\right)^2}\left(z-\tau_0\right) + O\left(\left(z-\tau_0\right)^2\right).
$$
Hence 
\begin{multline*}
F_{\tau_0}(z) = \frac{E_4\left(\tau_0\right)}{\Delta\left(\tau_0\right)j'\left(\tau_0\right)^2}\left(z-\tau_0\right)^{-2}\\
 + \left(-\frac{E_4\left(\tau_0\right)}{\Delta\left(\tau_0\right)}\cdot \frac{j''\left(\tau_0\right)}{j'\left(\tau_0\right)^3}+ \frac{\Delta\left(\tau_0\right)E_4'\left(\tau_0\right)-E_4\left(\tau_0\right)\Delta'\left(\tau_0\right)}{\Delta\left(\tau_0\right)^2j'\left(\tau_0\right)^2}\right)\left(z-\tau_0\right)^{-1} + O(1).
\end{multline*}

We have hence shown that the Laurent series coefficients of the principal parts of $F_{\tau_0}$ around $z=\tau_0$ are given by 
\begin{equation}\label{eqn:deflambda1}
\lambda_{-2}:=\frac{E_4\left(\tau_0\right)}{\Delta\left(\tau_0\right)j'\left(\tau_0\right)^2},
\end{equation}
\begin{equation}\label{eqn:deflambda2}
\lambda_{-1}:=-\frac{E_4\left(\tau_0\right)}{\Delta\left(\tau_0\right)}\cdot \frac{j''\left(\tau_0\right)}{j'\left(\tau_0\right)^3}+ \frac{\Delta\left(\tau_0\right)E_4'\left(\tau_0\right)-E_4\left(\tau_0\right)\Delta'\left(\tau_0\right)}{\Delta\left(\tau_0\right)^2j'\left(\tau_0\right)^2}.
\end{equation}

We now compare these with the Laurent series expansions of $Y_{-8, -2}(\tau_0,z)$ and $H_{10}\left(\tau_0,z\right)$.  
By (4a.9) of \cite{Pe1}, we have (where $v_0:=\im\left(\tau_0\right)$)
\begin{multline*}
Y_{-8,-2}\left(\tau_0,z\right)= -4i v_0\left(z-\overline{\tau_0}\right)^{8}X_{\tau_0}(z)^{-2} + O(1)= -4iv_0\left(z-\overline{\tau_0}\right)^{10}\left(z-\tau_0\right)^{-2}+O(1)\\
=-2\left(2iv_0\right)^{11}\left(z-\tau_0\right)^{-2} -40\left(2iv_0\right)^{10}\left(z-\tau_0\right)^{-1}+O(1).
\end{multline*}
Furthermore \eqref{eqn:zRes} gives that
$$
H_{10}(\tau_0,z) = -2\left(z-\tau_0\right)^{-1} +O(1).
$$
We conclude that
$$
F_{\tau_0}(z) = -\frac{\lambda_{-2}}{4\left(2iv_0\right)^{11}} Y_{-8,-2}\left(\tau_0,z\right) + \left( \frac{\lambda_{-2}}{40\left(2i v_0\right)^{10}}-\frac{\lambda_{-1}}{2}\right) H_{10}\left(\tau_0,z\right).
$$
Plugging Proposition \ref{prop:Ycoeff} and Theorem \ref{genFC} (1) into the above identity then yields the statement of the theorem.
\end{proof}


\begin{thebibliography}{99}
\bibitem{BeBi} B. Berndt and P. Bialek, \begin{it}Five formulas of Ramanujan arising from Eisenstein series\end{it}, Canadian Math. Soc. Conf. Proc. \textbf{15} (1995), 67--86.

\bibitem{BeBiYe} B. Berndt, P. Bialek, and A. Yee, \begin{it}Formulas of Ramanujan for the power series coefficients of certain quotients of Eisenstein series\end{it}, Int. Math. Res. Not. \textbf{2002} (2002), 1077--1109.


\bibitem{Bi}P. Bialek, \begin{it}Ramanujan's formulas for the coefficients in the power series expansions of certain modular forms\end{it}, Ph. D. thesis, University of Illinois at Urbana--Champaign, 1995.
\bibitem{123} J. Bruinier, G. van der Geer, and D. Zagier, The 1-2-3 of modular forms, Universitext, Springer-Verlag, Berlin, 2008.
\bibitem{ChowlaSelberg} S. Chowla and A. Selberg, \begin{it}On Epsteins zeta function\end{it}, Proc. Natl. Acad. Sci. (USA) \textbf{35} (1949), 371--374.
 \bibitem{HR1} G.  Hardy and S. Ramanujan, \emph{Une formule asymptotique pour le nombre des partitions de $n$},  Collected papers of Srinivasa Ramanujan,  239--241, AMS Chelsea Publ., Providence, RI, 2000.

\bibitem{HR2} G. Hardy and S. Ramanujan, \emph{Asymptotic formulae in combinatory analysis} Proc. London Math. Soc. \textbf{16} (1917), in Collected papers of Srinivasa Ramanujan, AMS Chelsea Publ, Providence, RI, 2000, 244.

\bibitem{HR3} G. Hardy and S. Ramanujan, \emph{On the coefficients in the expansions of certain modular functions}. Proc. Royal Soc. A \textbf{95} (1918), 144--155.

\bibitem{Le} J. Lehner, \begin{it}The Fourier coefficients of automorphic forms on horocyclic groups III\end{it}, Mich. Math J. \textbf{7} (1960), 65--74.


\bibitem{Pe1} H. Petersson, \begin{it}Konstruktion der Modulformen und der zu gewissen Grenzkreisgruppen geh\"origen automorphen Formen von positiver reeller Dimension und die vollst\"andige Bestimmung ihrer Fourierkoeffzienten\end{it}, S.-B. Heidelberger Akad. Wiss. Math. Nat. Kl. (1950), 415--474.

\bibitem{Rad} H. Rademacher, \emph{On the expansion of the partition function in a series}, Ann. of Math. \textbf{44} (1943), 416--422.

\bibitem{RZ} H. Rademacher and H.  Zuckerman,
\emph{On the Fourier coefficients of certain modular forms of positive dimension},  Ann. of Math.   \textbf{39} (1938),  433--462.

\bibitem{Ra} S. Ramanujan, \begin{it}On certain arithmetical functions,\end{it}, Trans. Cambridge Philos. Soc. \textbf{22} (1916), 159--184.

\bibitem{RaLost} S. Ramanujan, \begin{it}The lost notebook and other unpublished paper\end{it}, Narosa, New Delhi, 1988.

\bibitem{Zu1} H.  Zuckerman, \emph{On the coefficients
of certain modular forms belonging to subgroups of the modular group}, Trans. Amer. Math. Soc., \textbf{45} (1939), 298--321.

\bibitem{Zu2} H.  Zuckerman, \emph{On the expansions of certain modular forms of positive dimension}, Amer. J. Math. \textbf{62} (1940),  127--152.
\end{thebibliography}
\end{document}